\documentclass[10pt]{article}

\usepackage{amsmath, amsfonts, amsthm, amssymb}
\usepackage[all,arc,color]{xy}
\usepackage{float}
\usepackage{svg}
\svgsetup{extract=true,
    inkscapepath=. 
}

\usepackage[alphabetic,backrefs]{amsrefs}
 
 \usepackage[parfill]{parskip}
 \usepackage{anysize}
 \marginsize{1in}{1in}{1in}{1in}
 \usepackage{thmtools}

 \usepackage[colorlinks = true,
 linkcolor = blue,
 urlcolor = cyan,
 citecolor = red,
 anchorcolor = blue,
 pagebackref,
 hypertexnames=false]{hyperref}

\begingroup
 \makeatletter
 \@for\theoremstyle:=definition,remark,plain\do{%
 \expandafter\g@addto@macro\csname th@\theoremstyle\endcsname{%
 \addtolength\thm@preskip\parskip 
 }%
 }
 \endgroup

\newtheorem{theorem}{Theorem}
\numberwithin{theorem}{section}
\newtheorem{lemma}[theorem]{Lemma}
\newtheorem{proposition}[theorem]{Proposition}
\newtheorem{corollary}[theorem]{Corollary}
\newtheorem{remark}[theorem]{Remark}

\makeatletter
\let\c@equation\c@theorem
\makeatother
\numberwithin{equation}{section}

\begin{document}

\title{Non-squeezing and capacities for some calibrated geometries}

\author{Kain Dineen \\ {\it Department of Mathematics, University of Toronto} \\ \tt{kain.dineen@mail.utoronto.ca} \and Spiro Karigiannis \\ {\it Department of Pure Mathematics, University of Waterloo} \\ \tt{karigiannis@uwaterloo.ca} }

\date{June 1, 2026}

\maketitle

\begin{abstract}
It was shown by Barron--Shafiee that an analogue of Gromov's non-squeezing theorem holds for affine maps which preserve a power \(\omega^k\) of the symplectic form \(\omega\) on \(\mathbb{R}^{2n}\). In the present paper, we state and prove in two ways an improved version of their result which is closer to the classical affine non-squeezing theorem. One proof closely follows their argument, and the other consists of a reduction to the classical case by showing that, except for the case \(k = n\), every linear map that preserves \(\omega^k\) must be symplectic or anti-symplectic.

We then study when a calibration form satisfies an (affine) non-squeezing theorem. Particular focus is given to the special Lagrangian case, where we are able to establish an affine non-squeezing theorem for the holomorphic volume form \(\Omega = dz^1 \wedge \cdots \wedge dz^n\). The classical symplectic affine rigidity theorem states roughly that a non-singular linear map is symplectic or anti-symplectic if and only if it preserves the ``capacity'' of every ellipsoid. We establish an affine special Lagrangian version of this theorem under the necessary assumption that the map is complex-linear. We also discuss some natural future questions.
\end{abstract}

\tableofcontents

\section{Introduction}

One of the main questions of symplectic geometry is whether one symplectic manifold can be embedded into another (or whether one subset of a symplectic space can be embedded into another) by a symplectic map. One way to study this problem is to develop invariants of symplectic manifolds which obstruct embeddings. Gromov proved the seminal \emph{non-squeezing theorem} in~\cite{gromov}, which states that a ball of radius \(r\) can be symplectically embedded into a ``symplectic'' cylinder of radius \(R\) only if \(r \leq R\). 

Motivated by the embedding problem, Ekeland and Hofer in~\cite{ekelandhofer} developed such a symplectic invariant (of subsets of \(\mathbb{R}^{2n}\)) which they called a \emph{symplectic capacity}. They established the existence of a symplectic capacity and used it to give another proof of Gromov's non-squeezing theorem. Their capacities were then used to prove that symplectic (or anti-symplectic) maps can be characterized as those maps preserving the capacities of every ellipsoid. Then they used this to answer the rigidity problem for symplectic maps: what is the uniform limit of a sequence of symplectic maps? Section~\ref{sec:classical} of the present paper is devoted to a discussion of the classical non-squeezing theorem, where we focus on the affine case and on the applications in~\cite{ekelandhofer}. (Other notions of capacity exist. For instance, the notion of an \emph{analytic capacity} has been studied in complex geometry. See~\cite{BS}, for example.)

In~\cite{barron}, it was observed by Barron and Shafiee that a more general non-squeezing theorem, at least in the affine case, holds for the ``multisymplectic geometries'' given by the powers \(\omega^k\) of the symplectic form \(\omega\) on \(\mathbb{R}^{2n}\), and that there is a corresponding ``multisymplectic capacity'' for these geometries (still in the affine case). The motivation for the present paper is to find a generalization of the classical non-squeezing theorem and capacities (including the applications in~\cite{ekelandhofer}) to more general geometries along the lines of~\cite{barron}. Our generalizations in Section~\ref{sec:calibrations} (to general calibrations) and Section~\ref{sec:slag} (the special Lagrangian case) mostly remain in the affine setting.

We court two audiences: (i) symplectic geometers who may not be aware of the potentially fertile source of generalizations of some classical results of symplectic geometry to the setting of calibrated geometry; and (ii) experts in calibrated geometry who may not be intimately familiar with some of these classical symplectic results. Because of this dual audience, we choose to organize our paper as follows.

Section~\ref{sec:classical} is a review of some classical results, both to fix notation and to keep the paper reasonably self-contained. We briefly recall the non-linear version of Gromov's non-squeezing theorem in Section~\ref{section:2.1}, and in Section~\ref{section:2.2} we prove the affine version of the theorem. The affine non-squeezing theorem suggests a ``non-squeezing property'' for linear maps, which we prove characterizes symplectic and anti-symplectic maps. We then define the \emph{linear symplectic width} and prove that a linear map is symplectic or anti-symplectic if and only if it preserves the linear symplectic width of every ellipsoid (``affine rigidity''). In Section~\ref{section:2.3}, we show that every ellipsoid admits a symplectic normal form and we compute the linear symplectic width of an ellipsoid in terms of its normal form. In Section~\ref{section:2.4}, we assume that Gromov's non-squeezing theorem holds and we explore its consequences for the rigidity of symplectic geometry; in particular, we establish a non-linear version of the affine rigidity theorem.

In Section~\ref{sec:powers}, we study non-squeezing for the powers of the symplectic form. In Section~\ref{section:3.1} we state and prove the affine non-squeezing theorem for \(\omega^k\) discovered in~\cite[Theorem 2.16]{barron} and proceed to improve the bound in their result. One of the crucial technical ingredients in this proof is that the stabilizer of \(\omega^k\) is closed under transposition. This follows from the results of Section~\ref{section:3.2}, where we study the stabilizer in \(\mathrm{GL}(2n, \mathbb{R})\) of \(\omega^k\) and give a complete description of it in terms of the usual symplectic group. We apply this to give a second proof of the affine non-squeezing theorem for \(\omega^k\) by reduction to the symplectic case. In Section~\ref{section:3.3}, we apply our characterization of the stabilizer of \(\omega^k\) to show that \(\omega^k\) enjoys a non-linear non-squeezing theorem as well (again a consequence of the symplectic case), and then we give applications to symplectic manifolds (in particular we obtain a new proof of the main theorem of the recent paper~\cite{alizadeh}).

In Section~\ref{sec:calibrations}, motivated by the proof in Section~\ref{section:3.1} of the affine non-squeezing theorem for \(\omega^k\), we show that a \emph{calibration} on a real inner product space enjoys an affine non-squeezing theorem when its stabilizer is closed under transposition. We also show that this is not a necessary condition for a calibration to satisfy an affine non-squeezing theorem. In Section~\ref{section:4.2}, we briefly study the particular case of the \(\mathrm{G}_2\) calibration, where we show that it is very easy to prove that a non-linear non-squeezing theorem holds (due to the fact that the associative \(3\)-form $\varphi$ determines the metric.)

In Section~\ref{sec:slag}, we study the special Lagrangian case. Section~\ref{section:5.1} reviews the required linear-algebraic facts. In Section~\ref{section:5.2}, we show that for \(n \leq 3\), the special Lagrangian calibration \(\mathrm{Re}(dz^1 \wedge \cdots \wedge dz^n)\) on \(\mathbb{R}^{2n}\) enjoys the affine non-squeezing theorem. We then consider, for all \(n\), affine non-squeezing and rigidity for the complex-valued \(n\)-form \(dz^1 \wedge \cdots \wedge dz^n\). In particular, we establish an affine non-squeezing theorem for this form and a restricted version of the affine rigidity theorem for complex-linear maps.

We conclude with a discussion of some further questions in Section~\ref{sec:conclusion}.

\textbf{Acknowledgements.} Much of this work was completed while the first author earned a Master of Mathematics degree at the University of Waterloo under the supervision of the second author. The authors gratefully acknowledge useful discussions with Shengda Hu and Ben Webster, and thank the anonymous referees for very useful feedback on an earlier version of this article, which in particular allowed us to strengthen our Theorem~\ref{thm:slagrid}, as the complex-linearity assumption is in fact necessary.

\textbf{Data availability statement.} There were no datasets generated for this research project.

\textbf{Statement of interests.} All authors certify that they have no affiliations with or involvement in any organization or entity with any financial interest or non-financial interest in the subject matter or materials discussed in this manuscript.

\textbf{Funding declaration.} The research of the second author is partially supported by the Natural Sciences and Engineering Research Council of Canada (NSERC), grant number RGPIN-2025-03951.

\section{Classical non-squeezing and capacities} \label{sec:classical}

We discuss Gromov's non-squeezing theorem with a focus on the linear case. The presentation here mostly follows~\cite{mcduffsalamon} with some contributions from~\cite{ekelandhofer}. We give precise references for each of the individual statements. But as stated in the introduction, we present complete proofs to keep the paper self-contained, since the calibrated geometry community may be less familiar with this material.

\subsection{Gromov's non-squeezing theorem}
\label{section:2.1}

Consider \(\mathbb{R}^{2n}\) equipped with the standard symplectic form
\[
 \omega = dx^1 \wedge dy^1 + \cdots + dx^n \wedge dy^n.
\]
We denote by \(B(r)\) the closed ball of radius \(r\) centered at the origin in \(\mathbb{R}^{2n}\) and by \(Z(R)\) the closed symplectic cylinder
\[
 Z(R) = \{ (x_1, \dots, x_n, y_1, \dots, y_n) \in \mathbb{R}^{2n} : x_1^2 + y_1^2 \leq R^2 \}.
\]

\begin{theorem} \label{thm:gns}
 (Gromov's non-squeezing theorem) If there exists a symplectic embedding of \((B(r), \omega)\) into \((Z(R), \omega)\), then \(r \leq R\).
\end{theorem}

We could have stated Gromov's non-squeezing theorem for the corresponding open ball and open cylinder, but it will be more convenient for us to work with their closures.

It is crucial that the cylinder \(Z(R)\) is defined with respect to the symplectic splitting \(\mathbb{R}^2_{x_1, y_1} \times \mathbb{R}^{2n-2}\), since there is no analogue of the non-squeezing theorem for, say, the ball and the cylinder
\[
 B(R) \times \mathbb{R}^n = \{(x, y) \in \mathbb{R}^{2n} : |x|^2 \leq R^2\}
\]
defined with respect to the Lagrangian splitting \(\mathbb{R}_x^n \times \mathbb{R}_y^n\). Indeed, for every \(\delta > 0\), the map \((x, y) \mapsto (\delta x, \delta^{-1} y)\) is a symplectic embedding which takes the unit ball into \(B(\delta) \times \mathbb{R}^{n}\).

We will not prove Gromov's non-squeezing theorem in its full generality here, but we will prove the theorem for affine symplectic maps in the next section. Gromov's proof in~\cite{gromov} used the theory of pseudo-holomorphic curves. Another proof was given by Ekeland and Hofer in~\cite{ekelandhofer} using the construction of a particular kind of symplectic invariant known as a \emph{symplectic capacity}. Symplectic capacities and their applications will be the focus of the rest of this section.

\subsection{Affine non-squeezing, capacities, and rigidity}
\label{section:2.2}

We first establish some notation. We denote the symplectic group by 
\[
 \mathrm{Sp}(2n, \mathbb{R}) = \{ \Psi \in \mathrm{GL}(2n, \mathbb{R}) : \Psi^* \omega = \omega \}.
\]
We  denote the standard symplectic basis of \(\mathbb{R}^{2n}\) by \(e_1, \dots, e_n, f_1, \dots, f_n\), so that
\[
 \omega = e^1 \wedge f^1 + \cdots + e^n \wedge f^n.
\]
We will make use of the almost complex structure \(J\) on \(\mathbb{R}^{2n}\) associated to the symplectic structure \(\omega\) and the standard Euclidean metric \(g\), defined by \(\omega(x, y) = g(Jx, y)\). Note that
\[
 \Psi \in \mathrm{Sp}(2n, \mathbb{R}) \text{ if and only if } \Psi^T J \Psi = J.
\]
Before proving the non-squeezing theorem for affine symplectic maps, we need the following fact about the symplectomorphism group. In general, when we try to generalize the affine case of the non-squeezing theorem to other calibrations in later sections, the idea of certain groups being closed under transposition will become important.

\begin{lemma}
 The symplectic group is closed under transposition.
\end{lemma}
\begin{proof}
 If \(\Psi\) is symplectic, then so is \(\Psi^{-1}\). This implies that \((\Psi^{-1})^T J \Psi^{-1} = J\). Taking the inverses of both sides and using the fact that \(J^{-1} = -J\) gives \(\Psi J \Psi^T = J\). In other words, \((\Psi^T)^T J \Psi^T = J\), implying that \(\Psi^T\) is symplectic.
\end{proof}

We are now ready to state and prove the affine non-squeezing theorem.
\begin{theorem}[{\cite[Theorem 2.4.1]{mcduffsalamon}}] \label{thm:affns}
 (Affine non-squeezing) If there exists an affine symplectomorphism taking \((B(r), \omega)\) into \((Z(R), \omega)\), then \(r \leq R\).
\end{theorem}
\begin{proof}
 Suppose that \(\psi\) is an affine symplectomorphism such that \(\psi(B(r)) \subseteq Z(R)\), where \(\psi\) is given by \(\psi(z) = \Psi z + z_0\) for \(\Psi\) in \(\mathrm{Sp}(2n, \mathbb{R})\) and \(z_0\) in \(\mathbb{R}^{2n}\). Since the symplectic group is closed under transposition, \(\Psi^T\) is also symplectic, and it follows that
 \[
 1 = |\omega(e_1, f_1)| = |\omega(\Psi^T e_1, \Psi^T f_1)|.
 \]
 From the fact that \(\omega(\cdot, \cdot) = g(J \cdot, \cdot)\) and the Cauchy-Schwarz inequality, we have that 
 \[
 1 = |\omega(\Psi^T e_1, \Psi^T f_1)| \leq |J\Psi^Te_1| |\Psi^T f_1| = |\Psi^T e_1| |\Psi^T f_1|.
 \]
 So, either \(\Psi^T e_1\) or \(\Psi^T f_1\) has norm at least one. We may assume without loss of generality that \(|\Psi^T e_1| \geq 1\). Now, define \(z = \pm r \cdot \Psi^T e_1 / |\Psi^T e_1|\) with the sign to be determined. Since \(z \in B(r)\), it follows that \(\psi(z) \in Z(R)\), hence 
 \begin{align*}
 R^2 &\geq \langle e_1, \psi(z)\rangle^2 + \langle f_1, \psi(z) \rangle^2 \\
 &\geq \langle e_1, \Psi z + z_0 \rangle^2 \\
 &= [ \langle \Psi^T e_1, z \rangle + \langle e_1, z_0 \rangle ]^2 \\
 &= [ \pm r|\Psi^T e_1| + \langle e_1, z_0 \rangle ]^2.
 \end{align*}
 We choose the sign to be the same as that of \(\langle e_1, z_0 \rangle\). Hence \(R^2 \geq r^2 |\Psi^T e_1|^2 \geq r^2\), and so \(R \geq r\).
\end{proof}

The only place where we used the fact that \(\Psi\) (rather, \(\Psi^T\)) was symplectic was in the equality \(|\omega(e_1, f_1)| = |\omega(\Psi^T e_1, \Psi^T f_1)|\), so we see that the affine non-squeezing theorem holds for affine \emph{anti}-symplectomorphisms as well. We can also deduce the result for anti-symplectomorphisms from the result for symplectomorphisms by composing with an anti-symplectomorphism of the unit ball with itself. (This method is used in the proof of the upcoming Proposition~\ref{prop:antisymp}.)

A subset \(B\) of \(\mathbb{R}^{2n}\) is a \emph{linear} (resp.\ \emph{affine}) \emph{symplectic ball of radius \(r\)} if it is linearly (resp.\ affinely) symplectomorphic to \(B(r)\). Similarly, \emph{linear} (resp.\ \emph{affine}) \emph{symplectic cylinders} are subsets of \(\mathbb{R}^{2n}\) that are linearly (resp.\ affinely) symplectomorphic to \(Z(R)\) for some \(R > 0\). A consequence of the affine non-squeezing theorem is that the radius \(R\) is an affine symplectic invariant of a symplectic cylinder.

We say that a linear map \(\Psi \colon \mathbb{R}^{2n} \to \mathbb{R}^{2n}\) has the \emph{linear non-squeezing property} if, for every linear symplectic ball \(B\) of radius \(r\) and every linear symplectic cylinder \(Z\) of radius \(R\), the condition \(\Psi(B) \subseteq Z\) implies that \(r \leq R\). The affine non-squeezing theorem implies that every linear symplectomorphism has the linear non-squeezing property. Moreover, linear anti-symplectomorphisms also have the linear non-squeezing property. It is a perhaps surprising fact, which we now prove, that this property characterizes such maps.

\begin{theorem}[{\cite[Theorem 2.4.2]{mcduffsalamon}}] \label{thm:affrigid}
 (Affine rigidity) Let \(\Psi\) be a non-singular linear map from \(\mathbb{R}^{2n}\) to itself such that both \(\Psi\) and \(\Psi^{-1}\) have the linear non-squeezing property. Then \(\Psi\) is either symplectic or anti-symplectic.
\end{theorem}
\begin{proof}
 Suppose that \(\Psi\) is neither symplectic nor anti-symplectic. Then neither is \(\Psi^T\), so that \((\Psi^T)^*\omega \neq \omega\) and \((\Psi^T)^*\omega \neq -\omega\). The intersection of the two open and dense sets of pairs \((x, y)\) such that \(((\Psi^T)^*\omega)(x, y) \neq \omega(x, y)\) and \(((\Psi^T)^*\omega)(x, y) \neq -\omega(x, y)\) is non-empty, so we can find a pair \((u, v)\) which witnesses both conditions simultaneously:
 \[
 \omega(\Psi^T u, \Psi^T v) \neq \omega(u, v) \,\,\text{ and }\,\, \omega(\Psi^T u, \Psi^T v) \neq -\omega(u, v).
 \]
 This condition on \(u\) and \(v\) is an open condition, so we may perturb \(u\) and \(v\) to ensure that \(\omega(u, v) \neq 0\) and that \(\omega(\Psi^T u, \Psi^T v) \neq 0\), the latter requiring the non-singularity of \(\Psi\). By replacing \(\Psi\) by \(\Psi^{-1}\) if necessary, we may assume that \(|\omega(\Psi^Tu, \Psi^T v)| < |\omega(u, v)|\), and then finally by scaling \(u\) if necessary we may assume that
 \[
 0 < |\omega(\Psi^T u, \Psi^T v)| < \omega(u, v) = 1.
 \]
 Since \(\omega(u, v) = 1\), there exists a linear symplectomorphism \(\Phi\) such that \(\Phi e_1 = u\) and \(\Phi f_1 = v\). If we define \(\lambda = \sqrt{|\omega(\Psi^T u, \Psi^T v)|}\), then similarly there exists a linear symplectomorphism \(\Phi'\) such that \(\Phi'e_1 = \pm \lambda^{-1} \Psi^T u\) and \(\Phi' f_1 = \lambda^{-1} \Psi^T v\).

 Now, define \(A = (\Phi')^{-1} \Psi^T \Phi\). By construction, \(Ae_1 = \pm \lambda e_1\) and \(Af_1 = \lambda f_1\). It follows that \(A^T(B(1)) \subseteq Z(\lambda)\), for if \(z \in B(1)\), then
 \[
 \langle e_1, A^Tz \rangle^2 + \langle f_1, A^T z \rangle^2 = \langle \pm \lambda e_1, z \rangle^2 + \langle \lambda f_1, z \rangle^2 = \lambda^2( \langle e_1, z \rangle^2 + \langle f_1, z \rangle^2 ) \leq \lambda^2.
 \]
 Since \(\lambda < 1\), the map \(A^T = \Phi^T \Psi (\Phi'^T)^{-1}\) cannot have the linear non-squeezing property, implying that \(\Psi\) cannot have it either.
\end{proof}

Using the affine rigidity theorem, we can in fact characterize linear symplectomorphisms and anti-symplectomorphisms as exactly those maps which preserve a certain affine symplectic invariant known as the \emph{linear symplectic width}, or \emph{Gromov width} (see Theorem~\ref{thm:linchar}). The linear symplectic width of a (non-empty) subset \(A\) of \(\mathbb{R}^{2n}\) is defined to be 
\[
 w_L(A) := \sup\{ \pi r^2 : \Psi(B(r)) \subseteq A \text{ for some affine symplectomorphism } \Psi \}.
\]
We take \(\pi r^2\) in the definition of the linear symplectic width so that for \(\mathbb{R}^2\), the linear symplectic width and the area of a disk agree. The linear symplectic width satisfies the following three properties:
\begin{itemize}
 \item \emph{Monotonicity}: If \(A\) and \(B\) are subsets of \(\mathbb{R}^{2n}\) and if \(\Psi(A) \subseteq B\) for some affine symplectomorphism \(\Psi\), then \(w_L(A) \leq w_L(B)\).

 \item \emph{Conformality}: \(w_L(\lambda A) = \lambda^2 w_L(A)\).

 \item \emph{Non-triviality}: \(w_L(B(r)) = w_L(Z(r)) = \pi r^2\).
\end{itemize}
The monotonicity and conformality properties of \(w_L\) are immediate from the definition, and the non-triviality property is proved by the following chain of inequalities:
\begin{equation}
 \label{eq:2.1}
 \pi r^2 \leq w_L(B(r)) \leq w_L(Z(r)) \leq \pi r^2.
\end{equation}
The first inequality is immediate from the definition, the second follows from monotonicity and the fact that \(B(r) \subseteq Z(r)\), and the third follows from the affine non-squeezing theorem.

The monotonicity property implies that the linear symplectic width of a subset of \(\mathbb{R}^{2n}\) is an affine symplectic invariant, as the following result shows.
\begin{proposition}
 \label{prop:linwidth}
 If \(\Psi\) is an affine symplectomorphism, then \(w_L(A) = w_L(\Psi(A))\) for every subset \(A\) of \(\mathbb{R}^{2n}\).
\end{proposition}
\begin{proof}
 By monotonicity, \(w_L(A) \leq w_L(\Psi(A))\). Since \(\Psi^{-1}\) is also an affine symplectomorphism, we have by monotonicity again that \(w_L(\Psi(A)) \leq w_L(\Psi^{-1}(\Psi(A))) = w_L(A)\).
\end{proof}

In particular, any linear symplectic ball (resp.\ cylinder) of radius \(r\) (resp.\ \(R\)) has linear symplectic width \(\pi r^2\) (resp.\ \(\pi R^2\)). We will also need the following fact that the linear symplectic width is also an affine anti-symplectic invariant.

\begin{proposition} \label{prop:antisymp}
 If \(\Psi\) is an affine anti-symplectomorphism, then \(w_L(A) = w_L(\Psi(A))\) for every subset \(A\) of \(\mathbb{R}^{2n}\).
\end{proposition}
\begin{proof}
 As above, it suffices to show that \(w_L(A) \leq w_L(\Psi(A))\). For this, we need to show that if \(B(r)\) is taken into \(A\) by an affine symplectomorphism \(\Phi\), then \(\pi r^2 \leq w_L(\Psi(A))\). Let \(T\) be the anti-symplectomorphism given by \((x, y) \mapsto (-x, y)\), so that \(T(B(r)) = B(r)\). It follows that \(\Psi \circ \Phi \circ T\) is an affine symplectomorphism taking \(B(r)\) into \(\Psi(A)\), and so \(\pi r^2 \leq w_L(\Psi(A))\).
\end{proof}

The linear symplectic width allows us to give a characterization of linear symplectomorphisms and anti-symplectomorphisms as those linear maps that preserve the linear symplectic width. Specifically, we do not need our map to preserve the width of \emph{every} subset of \(\mathbb{R}^{2n}\); we will prove that it suffices for it to preserve the width of (closed) ellipsoids centered at the origin. (In particular, as we show in the next section, the linear symplectic width of such an ellipsoid is easy to compute.)

\begin{theorem}[{\cite[Theorem 2.4.4]{mcduffsalamon}}] \label{thm:linchar}
 Let \(\Psi \colon \mathbb{R}^{2n} \to \mathbb{R}^{2n}\) be a linear map. The following are equivalent:
 \begin{enumerate}
 \item \(\Psi\) preserves the linear symplectic width of every ellipsoid centered at the origin.

 \item \(\Psi\) is symplectic or anti-symplectic.
 \end{enumerate}
\end{theorem}
\begin{proof}
 We first prove that 2 implies 1. If \(\Psi\) is symplectic, then Proposition~\ref{prop:linwidth} implies that it preserves the width of every ellipsoid centered at the origin, and if \(\Psi\) is anti-symplectic, then we are done by Proposition~\ref{prop:antisymp}.

 Let us now prove that 1 implies 2. We begin by proving that \(\Psi\) is invertible. If \(\Psi\) is not invertible, then \(\Psi(B(1))\) is contained in a proper linear subspace of \(\mathbb{R}^{2n}\). Thus, any ball which embeds into \(\Psi(B(1))\) must embed into a proper subspace of \(\mathbb{R}^{2n}\). This is impossible for a ball of positive radius, since the image of any such ball under an embedding must have a non-empty interior, so it follows that \(w_L(\Psi(B(1))) = 0\). But the normalization property implies that \(w_L(\Psi(B(1))) = w_L(B(1)) = \pi\), a contradiction.
 
 We now show that \(\Psi\) and \(\Psi^{-1}\) have the linear non-squeezing property. Let \(B\) be a linear symplectic ball of radius \(r\) and let \(Z\) be a linear symplectic cylinder of radius \(R\) such that \(\Psi(B) \subseteq Z\). Since \(B\) is an ellipsoid centered at the origin, we have that \(w_L(B) = w_L(\Psi(B))\) by assumption, and so
 \[
 \pi r^2 = w_L(B) = w_L(\Psi(B)) \leq w_L(Z) = \pi R^2.
 \]
 Note that \(\Psi^{-1}\) also preserves the linear symplectic width of every ellipsoid \(E\) centered at the origin, since
 \[
 w_L(\Psi^{-1}(E)) = w_L(\Psi(\Psi^{-1}(E))) = w_L(E).
 \]
 (The first equality follows from the fact that \(\Psi^{-1}(E)\) is an ellipsoid centered at the origin and the fact that \(\Psi\) preserves the width of all such ellipsoids.) The same argument as before then shows that \(\Psi^{-1}\) has the linear non-squeezing property. We conclude by the affine rigidity theorem (Theorem~\ref{thm:affrigid}) that \(\Psi\) must then be symplectic or anti-symplectic.
\end{proof}

Note that by the conformality property \(w_L(\lambda E) = \lambda^2 w_L(E)\), a linear map preserves the capacity of all ellipsoids centered at the origin if and only if it preserves the capacity of all sufficiently small ellipsoids centered at the origin. (We will use this in the proof of Theorem~\ref{thm:limit} below.)

\subsection{The symplectic width of an ellipsoid}
\label{section:2.3}

Theorem~\ref{thm:linchar} suggests that we should try to understand the linear symplectic widths of (closed) ellipsoids centered at the origin. We will restrict here to closed ellipsoids centered at the origin. Generally, we want to try to find linear symplectic invariants of ellipsoids. In fact, we can find a complete linear symplectic invariant of an ellipsoid known as its \emph{symplectic spectrum}. We begin with the following result from symplectic linear algebra, originally proved in~\cite{Will-1}. (See also~\cite{Will-2}.) 

\begin{theorem} \label{thm:williamson}
 (Williamson's theorem.) Let \(M\) be a real symmetric positive-definite \(2n \times 2n\) matrix. There exists \(S \in \mathrm{Sp}(2n, \mathbb{R})\) such that \(S^T M S = \Lambda \oplus \Lambda\), where \(\Lambda\) is a diagonal \(n \times n\) matrix with positive entries. Moreover, up to a reordering, the diagonal entries of \(\Lambda\) do not depend on \(S\).
\end{theorem}
\begin{proof}
 The matrix \(M^{-\frac{1}{2}} J M^{-\frac{1}{2}}\) is anti-symmetric since \(M^{-\frac{1}{2}}\) is symmetric and \(J\) is anti-symmetric, and it is invertible since \(M^{-\frac{1}{2}}\) and \(J\) are, so there exists a matrix \(Q \in \mathrm{O}(2n)\) such that
 \[
 Q M^{-\frac{1}{2}} J M^{-\frac{1}{2}} Q^T = J_2 \otimes D, \qquad J_2 = \begin{pmatrix} 0 & -1 \\ 1 & 0 \end{pmatrix},
 \]
 where \(\otimes\) is the Kronecker product and \(D\) is an \(n \times n\) diagonal matrix with positive entries. Note that \(J = J_2 \otimes I_n\), so that in particular
 \[
 (I_2 \otimes D^{-\frac{1}{2}})(Q M^{-\frac{1}{2}} J M^{-\frac{1}{2}} Q^T)(I_2 \otimes D^{-\frac{1}{2}}) = J.
 \]
 Let \(\tilde{D} = I_2 \otimes D\), so that
 \[
 (M^{-\frac{1}{2}} Q^T \tilde{D}^{-\frac{1}{2}})^T J (M^{-\frac{1}{2}} Q^T \tilde{D}^{-\frac{1}{2}}) = J.
 \]
 This means that \(S = M^{-\frac{1}{2}} Q^T \tilde{D}^{-\frac{1}{2}}\) is a symplectic matrix. Also,
 \[
 S^TMS = \tilde{D}^{-\frac{1}{2}} Q M^{-\frac{1}{2}} M M^{-\frac{1}{2}} Q^T \tilde{D}^{-\frac{1}{2}} = \tilde{D}^{-1} = \begin{pmatrix} D^{-1} & 0 \\ 0 & D^{-1} \end{pmatrix},
 \]
 which proves existence. For uniqueness, suppose that \(S^T(\Lambda \oplus \Lambda)S = \Lambda' \oplus \Lambda'\) for some symplectic \(S\). Since \(SJS^T = J\), this identity is equivalent to \(S^{-1} J(\Lambda \oplus \Lambda) S = J(\Lambda' \oplus \Lambda')\). So, \(J(\Lambda \oplus \Lambda)\) and \(J(\Lambda' \oplus \Lambda')\) have the same eigenvalues. The eigenvalues of \(J(\Lambda \oplus \Lambda)\) are \(\pm i \Lambda_{jj}\) and likewise for \(J(\Lambda' \oplus \Lambda')\), so \(\Lambda\) and \(\Lambda'\) agree up to a permutation of the diagonal.
\end{proof}

Now, given an $n$-tuple \(r = (r_1, \dots, r_n)\) with \(0 < r_1 \leq \cdots \leq r_n\), we define
\[
 E(r) := \left\{ (x, y) \in \mathbb{R}^{2n} : \sum_{j=1}^n \frac{x_j^2 + y_j^2}{r_j^2} \leq 1 \right\}.
\]
We will now use Williamson's theorem to prove that every closed ellipsoid centered at the origin is linearly symplectomorphic to one of the form \(E(r)\) for a unique such \(r\).

\begin{proposition}[{\cite[Lemma 2.4.6]{mcduffsalamon}}] \label{prop:normalform}
 Given any closed ellipsoid \(E\) centered at the origin, there exists \(S \in \mathrm{Sp}(2n, \mathbb{R})\) and a unique $n$-tuple \(r = (r_1, \dots, r_n)\) with \(0 < r_1 \leq \cdots \leq r_n\) such that \(E = S(E(r))\).
\end{proposition}
\begin{proof}
 We can find a real symmetric positive-definite \(2n \times 2n\) matrix \(M\) such that
 \[
 E = \{z \in \mathbb{R}^{2n} : g(z, Mz) \leq 1\}.
 \]
 By Theorem~\ref{thm:williamson}, we can find an \(S\) in \(\mathrm{Sp}(2n)\) and a diagonal \(n \times n\) matrix \(\Lambda\) such that \(S^TMS = \Lambda \oplus \Lambda\). We may assume that the diagonal entries of \(\Lambda\) are \emph{decreasing}, which in particular uniquely determines \(\Lambda\). If \(z = (x, y) \in \mathbb{R}^{2n}\), then
 \[
 g(Sz, MSz) = g(z, (\Lambda \oplus \Lambda)z) = \sum_{j=1}^n \lambda_j^2(x_j^2 + y_j^2) = \sum_{j=1}^n \frac{x_j^2 + y_j^2}{(1/\lambda_j)^2}.
 \]
 So, if we define \(r_j = 1/\lambda_j\), then \(0 < r_1 \leq \cdots \leq r_n\) by the assumption that \(\lambda_1 \geq \cdots \geq \lambda_n\), and it follows that \(S^{-1} E = E(r)\) for \(r = (r_1, \dots, r_n)\).
\end{proof}

The $n$-tuple \(r\) associated to \(E\) by Proposition~\ref{prop:normalform} is called the \emph{symplectic spectrum} of \(E\). An immediate consequence of this, as promised at the start of this section, is that two closed ellipsoids are linearly symplectomorphic \emph{if and only if} they have the same symplectic spectrum. We conclude this section by computing the linear symplectic width of an ellipsoid in terms of its symplectic spectrum.

\begin{theorem}[{\cite[Theorem 2.4.8]{mcduffsalamon}}] \label{thm:spec}
 Let \(E\) be a closed ellipsoid centered at the origin in \(\mathbb{R}^{2n}\). If \(r = (r_1, \dots, r_n)\) is the symplectic spectrum of \(E\), then the linear symplectic width of \(E\) is given by
 \[
 w_L(E) = \pi r_1^2.
 \]
\end{theorem}
\begin{proof}
 Since linear symplectomorphisms preserve the linear symplectic width, and since \(E\) is linearly symplectomorphic to \(E(r)\), it suffices to show that \(w_L(E(r)) = \pi r_1^2\). Since
 \[
 B(r_1) \subseteq E(r) \subseteq Z(r_1),
 \]
 it follows from the monotonicity property of \(w_L\) that
 \[
 \pi r_1^2 = w_L(B(r_1)) \leq w_L(E(r)) \leq w_L(Z(r_1)) = \pi r_1^2. \qedhere
 \]
\end{proof}

\subsection{Non-linear symplectic capacities and rigidity}
\label{section:2.4}

We now discuss generalizing the content of Section~\ref{section:2.3} to the non-linear case. In particular, we show that Theorem~\ref{thm:linchar} admits a generalization to a theorem which characterizes \emph{non-linear} symplectomorphisms and anti-symplectomorphisms.

To generalize Theorem~\ref{thm:linchar}, we need to generalize the linear symplectic width. We define a \emph{symplectic capacity on \(\mathbb{R}^{2n}\)} to be a map
\[
 c \colon \{\text{non-empty subsets of } \mathbb{R}^{2n}\} \to [0, \infty]
\]
satisfying the following three non-linear analogues of the properties of \(w_L\):
\begin{itemize}
 \item \emph{Monotonicity}: Let $A, B$ be non-empty subsets of $\mathbb{R}^{2n}$. If there exists a symplectic embedding \(\psi \colon A \hookrightarrow \mathbb{R}^{2n}\) such that \(\psi(A) \subseteq B\), then \(c(A) \leq c(B)\).

 \item \emph{Conformality}: \(c(\lambda A) = \lambda^2 c(A)\).

 \item \emph{Non-triviality}: \(c(B(1)) > 0\) and \(c(Z(1)) < \infty\).
\end{itemize}
(A symplectic embedding \(A \hookrightarrow \mathbb{R}^{2n}\) is one which extends to a symplectic embedding of an open neighbourhood of \(A\).) We note that the non-triviality axiom for a symplectic capacity above is slightly weaker than the non-triviality property for the linear symplectic width. We say that a symplectic capacity \(c\) is \emph{normalized} if \(c(B(1)) = c(Z(1)) = \pi\). It follows from the conformality axiom that \(c(B(r)) = \pi r^2\) and \(c(Z(R)) = \pi R^2\) for any normalized capacity \(c\).

It is not clear from the definition that symplectic capacities exist, let alone normalized ones. Following the linear case, we define the \emph{symplectic width} (or \emph{Gromov width}) of a non-empty subset \(A\) of \(\mathbb{R}^{2n}\) to be
\[
 w(A) := \sup\{ \pi r^2 : \psi(B(r)) \subseteq A \text{ for some symplectic embedding } \psi \colon B(r) \to \mathbb{R}^{2n} \}.
\]
Just as in the linear case, it is immediate from the definition that \(w\) satisfies the monotonicity and conformality properties. The same argument (see~\eqref{eq:2.1}) that proved that the linear symplectic width had the non-triviality property proves that \(w\) has the non-triviality property and is normalized, but the application of the affine non-squeezing theorem must be replaced by an application of Gromov's non-squeezing theorem (Theorem~\ref{thm:gns}).

The converse is true as well: if a normalized capacity on \(\mathbb{R}^{2n}\) exists, then Gromov's non-squeezing theorem holds. Indeed, if \(c\) is any normalized capacity, and if the ball \(B(r)\) embeds symplectically into the cylinder \(Z(R)\), then
\[
 \pi r^2 = c(B(1))r^2 = c(B(r)) \leq c(Z(R)) = c(Z(1))R^2 = \pi R^2.
\]
In~\cite{ekelandhofer}, Ekeland and Hofer introduced the general notion of a symplectic capacity on \(\mathbb{R}^{2n}\) and gave a proof of Gromov's non-squeezing theorem by constructing a normalized capacity on \(\mathbb{R}^{2n}\). A variant of the Gromov width in particular featured in Gromov's original paper~\cite{gromov} as the \emph{radius}, defined for a symplectic manifold \((M, \omega)\) as the supremum of the radii of balls which embed symplectically into \((M, \omega)\). (One can also define capacities on symplectic manifolds as opposed to subsets of \(\mathbb{R}^{2n}\), but we will focus only on capacities on \(\mathbb{R}^{2n}\).)

As another approach, rather than measuring the size of a subset of \(\mathbb{R}^{2n}\) by embedding progressively larger balls into it, we can try to measure the size of a subset by embedding it into progressively smaller cylinders. That is, we can define
\[
 \overline{w}(A) := \inf \{\pi R^2 : \psi(A) \subseteq Z(R) \text{ for some symplectic embedding } \psi \colon A \to \mathbb{R}^{2n}\}.
\]
The proof that \(\overline{w}\) is a normalized capacity is nearly identical to the proof that \(w\) is. Moreover, if \(c\) is any normalized capacity on \(\mathbb{R}^{2n}\), then we have that \(w \leq c \leq \overline{w}\). Indeed, if \(B(r) \hookrightarrow A\) and \(A \hookrightarrow Z(R)\) for some \(r, R > 0\) and some subset \(A\) of \(\mathbb{R}^{2n}\), then
\[
 \pi r^2 = c(B(r)) \leq c(A) \leq c(Z(R)) = \pi R^2.
\]
Thus \(w\) and \(\overline{w}\) are the ``extreme'' normalized capacities. In particular, if \(w(A) = \overline{w}(A)\) for some subset \(A\) of \(\mathbb{R}^{2n}\), then every normalized capacity attains the same value on \(A\). In fact, we have already seen a large class of subsets of \(\mathbb{R}^{2n}\) for which this is the case.

\begin{theorem}[{\cite[Example 12.1.7]{mcduffsalamon}}] \label{thm:ellcap}
 Let \(c\) be any normalized capacity on \(\mathbb{R}^{2n}\). If \(E\) is any closed ellipsoid in \(\mathbb{R}^{2n}\), then \(c(E)\) is equal to the linear symplectic width of \(E\). (In particular, every normalized capacity takes the same value on any given ellipsoid.)
\end{theorem}
\begin{proof}
 After a translation to the origin, the exact same proof as that for Theorem~\ref{thm:spec} goes through with \(c\) in place of \(w_L\), since the monotonicity property for \(c\) implies that any linear symplectomorphism taking \(E\) to \(E(r)\) will also preserve \(c\).
\end{proof}

As a second application of the existence of symplectic capacities, we prove the following theorem due to Ekeland and Hofer. We then deduce a non-linear analogue of Theorem~\ref{thm:linchar}. In what follows, we fix a normalized capacity \(c\) on \(\mathbb{R}^{2n}\). 

\begin{theorem}[{\cite[Theorem 5]{ekelandhofer}}] \label{thm:limit}
  Let \((\varphi_k)\) be a sequence of continuous maps \(B(\rho) \to \mathbb{R}^{2n}\) converging uniformly to \(\varphi \colon B(\rho) \to \mathbb{R}^{2n}\). Suppose that each \(\varphi_k\) preserves the capacity of every ellipsoid centered at the origin in \(B(\rho)\). If \(\varphi\) is differentiable at \(0\), then \(\varphi'(0)\) is symplectic or anti-symplectic.
\end{theorem}
\begin{proof}
 We may assume that \(\varphi(0) = 0\). We will make use of the maximal capacity \(\overline{w}\) defined above. Let \(E \subset B(\rho)\) be a fixed ellipsoid centered at the origin. Given \(\varepsilon > 0\), the definition of \(\overline{w}(\varphi(E))\) ensures that there exists an \(r\) and a symplectic embedding \(\psi \colon \varphi(E) \to \mathbb{R}^{2n}\) such that
 \[
 \psi(\varphi(E)) \subseteq Z(r) \qquad \text{and} \qquad \overline{w}(\varphi(E)) \leq \pi r^2 < \overline{w}(\varphi(E)) + \varepsilon
 \]
 Choose an ellipsoid \(\tilde{E}\) (perhaps not centered at the origin) such that \(\psi(\varphi(E)) \subset \mathrm{int}(\tilde{E})\) and \(\pi r^2 \leq \overline{w}(\tilde{E}) \leq \overline{w}(\varphi(E)) + \varepsilon\). By uniform convergence, it follows that \(\psi(\varphi_k(E)) \subset \tilde{E}\) for all large enough \(k\), and so
 \[
 \overline{w}(E) = c(E) = c(\psi(\varphi_k(E))) \leq \overline{w}(\psi(\varphi_k(E))) \leq \overline{w}(\tilde{E}) \leq \overline{w}(\varphi(E)) + \varepsilon.
 \]
 Since \(\varepsilon > 0\) was arbitrary, we see that \(\overline{w}(E) \leq \overline{w}(\varphi(E))\) for every ellipsoid \(E \subset B(\rho)\). So, for \(t \in (0, 1)\), it follows from this and the conformality property that
 \[
 \overline{w}(E) = \frac{1}{t^2}\overline{w}(tE) \leq \frac{1}{t^2} \overline{w}(\varphi(tE)) = \overline{w}\left(\frac{1}{t}\varphi(tE)\right).
 \]
 Since \(t^{-1}\varphi(t \cdot)\) converges uniformly to \(\varphi'(0)\) as \(t \downarrow 0\), we conclude by the same argument (with \(t^{-1}\varphi(t\cdot)\) and \(\varphi'(0)\) in place of \(\varphi_k\) and \(\varphi\), respectively) that 
 \[
 \lim_{t\downarrow 0} \overline{w}\left(\frac{1}{t} \varphi(tE)\right) \leq \overline{w}(\varphi'(0)E).
 \]
 As a result,
 \begin{equation}
 \label{eq:2.2}
 \overline{w}(E) \leq \overline{w}(\varphi'(0)E)
 \end{equation}
 for every ellipsoid \(E \subset B(\rho)\). It follows from~\eqref{eq:2.2} as in the proof of Theorem~\ref{thm:linchar} that \(\varphi'(0)\) is invertible. In particular, there exists a \(d > 0\) such that
 \[
 |\varphi'(0)x| \geq d|x|, \qquad x \in \mathbb{R}^{2n}.
 \]
 We now show the reversed inequality in~\eqref{eq:2.2}. Since \(\varphi\) is differentiable at \(0\), we have that
 \[
 |\varphi(x) - \varphi'(0)x| \leq \varepsilon(|x|)|x|, \qquad x \in B(\rho),
 \]
 where \(\varepsilon \colon [0, \rho] \to [0, \infty)\) is defined by \(\varepsilon(0) = 0\) and
 \[
 \varepsilon(r) = \sup\left\{\frac{|\varphi(x) - \varphi'(0)x|}{|x|} : 0 < |x| \leq r\right\}, \qquad r \in (0, \rho].
 \]
 (Note that \(\varepsilon\) is increasing and continuous.) It follows that for every \(\delta > 0\), there is a \(k(\delta)\) such that for \(k \geq k(\delta)\), 
 \begin{equation}
 \label{eq:2.3}
 |\varphi_k(x) - \varphi'(0)x| \leq \varepsilon(|x|)|x| + \delta, \qquad x \in B(\rho).
 \end{equation}
 Fix \(\gamma > 0\). We now make the following claim:
 
 \emph{There exists a \(\tau \in (0, (1 + \gamma)^{-1})\) and an integer \(k\) such that for every \(z\) in \(E\), the equation}
 \begin{equation}
 \label{eq:2.4}
 t \varphi_k((1 + \gamma)\tau x) + (1 - t) \varphi'(0)((1 + \gamma)\tau x) = \varphi'(0)(\tau z)
 \end{equation}
 \emph{has no solutions \((x, t) \in \partial E \times [0, 1]\).}

 We will prove this claim by contradiction. Suppose that for every \(\tau\) and \(k\), there are \(z\), \(x\), and \(t\) as above such that~\eqref{eq:2.4} holds. We first rewrite~\eqref{eq:2.4} as
 \[
 \varphi'(0)(\tau z - (1 + \gamma)\tau x) = t\left[\varphi_k((1+\gamma)\tau x) - \varphi'(0)((1+\gamma)\tau x)\right].
 \]
 It follows that
 \begin{align*}
 |\varphi_k((1+\gamma)\tau x) - \varphi'(0)((1 + \gamma)\tau x)| &\geq |\varphi'(0)(\tau z - (1 + \gamma)\tau x)| \\
 &\geq \tau d |z - (1 + \gamma)x| \\
 &\geq \tau d C
 \end{align*}
 where \(C > 0\) is a constant depending only on \(E\) and \(\gamma\). (We can take \(C\) to be the distance between \(E\) and \((1 + \gamma)\partial E\).) If we choose \(\delta = \varepsilon((1 + \gamma)\tau\rho)(1 + \gamma)\tau\rho\) and \(k \geq k(\delta)\), then, by~\eqref{eq:2.3} and the fact that \(\varepsilon\) is increasing, it follows that
 \begin{align*}
 0 < d\tau C &\leq \varepsilon(|(1 + \gamma)\tau x|)|(1 + \gamma)\tau x| + \delta \\
 &\leq 2\varepsilon((1 + \gamma)\tau \rho)(1 + \gamma) \tau \rho.
 \end{align*}
 Since \(\varepsilon\) is continuous and \(\varepsilon(0) = 0\), it is possible to choose \(\tau\) to be small enough so that this inequality does not hold. This contradiction proves the claim.

We now use the above claim to finish the proof. With \(\tau\) and \(k\) as above, we will show that
 \begin{equation}
 \label{eq:2.5}
 \varphi'(0)(\tau E) \subset \varphi_k((1 + \gamma)\tau E).
 \end{equation}
 Suppose that this were not the case, so that there is a \(z\) in \(E\) for which the equation \(\varphi'(0)(\tau z) = \varphi_k((1 + \gamma)\tau x)\) has no solution \(x \in E\). Define a map \(H \colon \partial E \times [0, 1] \to S^{2n-1}\) by
 \[
 H(x, t) = \frac{t \varphi_k((1 + \gamma)\tau x) + (1 - t) \varphi'(0)((1 + \gamma)\tau x) - \varphi'(0)(\tau z)}{|t \varphi_k((1 + \gamma)\tau x) + (1 - t) \varphi'(0)((1 + \gamma)\tau x) - \varphi'(0)(\tau z)|}.
 \]
 The map \(H\) is well-defined by the claim just proven. The assumption we just made implies that \(H(\cdot, 1)\) extends to \(E\), which implies that \(H(\cdot, 1)\) must have Brouwer degree zero. The maps \(H(\cdot, 1)\) and \(H(\cdot, 0)\) are homotopic, so they must have the same degree. But \(H(\cdot, 0)\) is given by
 \[
 H(x, 0) = \frac{\varphi'(0)((1 + \gamma)x - z)}{|\varphi'(0)((1 + \gamma)x - z)|} = \frac{\varphi'(0)(x - (1 + \gamma)^{-1}z)}{|\varphi'(0)(x - (1 + \gamma)^{-1}z)|},
 \]
 so it must have degree \(1\) or \(-1\). This contradiction shows that the inclusion in~\eqref{eq:2.5} must hold. It follows that
 \[
 \tau^2 \overline{w}(\varphi'(0)E) = \overline{w}(\varphi'(0)(\tau E)) \leq \overline{w}(\varphi_k((1 + \gamma)\tau E)) = (1 + \gamma)^2 \tau^2 \overline{w}(E).
 \]
 Since \(\gamma > 0\) was arbitrary, it follows that \(\overline{w}(E) \leq \overline{w}(\varphi'(0)E)\). Combining this with~\eqref{eq:2.2}, we conclude that
 \[
 \overline{w}(E) = \overline{w}(\varphi'(0)E), \qquad \text{ for every ellipsoid } E \subset B(\rho).
 \]
 The linearity of \(\varphi'(0)\) and the conformality property of \(\overline{w}\) imply that \(\varphi'(0)\) preserves the \(\overline{w}\)-capacity of \emph{every} ellipsoid in \(\mathbb{R}^{2n}\) centered at the origin. It follows from Theorem~\ref{thm:ellcap} that \(\varphi'(0)\) preserves the linear symplectic width of every such ellipsoid, and we conclude from Theorem~\ref{thm:linchar} that \(\varphi'(0)\) is either symplectic or anti-symplectic.
\end{proof}

Our first application of Theorem~\ref{thm:limit} is to prove the following non-linear analogue of Theorem~\ref{thm:linchar}.

\begin{theorem}[{\cite[Theorem 4]{ekelandhofer}}] \label{thm:nonlinchar}
 Let \(\varphi \colon \mathbb{R}^{2n} \to \mathbb{R}^{2n}\) be a smooth map. The following are equivalent.
 \begin{enumerate}
 \item \(\varphi\) preserves the capacity of every ellipsoid.
 \item \(\varphi^*\omega = \pm \omega\).
 \end{enumerate}
\end{theorem}
\begin{proof}
 We begin by proving that 2 implies 1. If \(\varphi\) is symplectic, then we have already seen that \(\varphi\) preserves the capacity of every subset of \(\mathbb{R}^{2n}\). If \(\varphi\) is anti-symplectic and if \(E\) is an ellipsoid in \(\mathbb{R}^{2n}\), then we can find an anti-symplectomorphism \(\psi\) such that \(\psi(E) = E\), and then it follows from the symplectic case that \(c(\varphi(E)) = c(\varphi\psi(E)) = c(E)\). (Such a \(\psi\) can be constructed by using Proposition~\ref{prop:normalform}, similar to the construction of the map \(T\) in the proof of Proposition~\ref{prop:antisymp}.)

 Now, suppose that \(\varphi\) preserves the capacity of every ellipsoid. We want to prove that for every \(z\) in \(\mathbb{R}^{2n}\), the linear map \(\varphi'(z)\) satisfies \(\varphi'(z)^*\omega = \pm \omega\). (The sign will be independent of \(z\) by continuity.) By composing with translations, we may assume without loss of generality that \(z = 0\) and \(\varphi(z) = 0\). For \(t > 0\), define
 \[
 \varphi_t(x) = \frac{1}{t}\varphi(tx).
 \]
 The conformality property of \(c\) immediately implies that each \(\varphi_t\) preserves the capacity of every ellipsoid centered at the origin. Since \((\varphi_t)\) converges uniformly to \(\varphi'(0)\) as \(t \downarrow 0\), it follows from Theorem~\ref{thm:limit} that \(\varphi'(0)\) is either symplectic or anti-symplectic.
\end{proof}

As a second application of Theorem~\ref{thm:limit}, we have the following \emph{rigidity} result, due to Eliashberg~\cite{eliashberg} and Gromov~\cite{gromovpdr}. We follow the proof due to Ekeland and Hofer in~\cite{ekelandhofer}.

\begin{theorem}
 \label{thm:c0closure}
 For any symplectic manifold, the group of symplectomorphisms is closed in the group of diffeomorphisms with respect to the \(C^0\)-compact-open topology.
\end{theorem}
In other words, if a sequence of symplectomorphisms converges uniformly on compact sets to a diffeomorphism, then the limit is actually a symplectomorphism. We stress that no assumptions on the derivatives of the symplectomorphisms are needed.
\begin{proof}
 This is a local statement, so by Darboux's theorem we may assume that we have a sequence \((\varphi_k)\) of symplectic embeddings of the ball \(B(\rho)\) into \(\mathbb{R}^{2n}\) which converges uniformly to a smooth embedding \(\varphi \colon B(\rho) \to \mathbb{R}^{2n}\) and that we need to prove that \(\varphi'(0)\) is symplectic. The maps \(\varphi_k\) are symplectic, so in particular they preserve the capacity of every ellipsoid, and then Theorem~\ref{thm:limit} implies that \(\varphi'(0)\) is either symplectic or anti-symplectic. As we will now show, the fact that each \(\varphi_k\) is actually a symplectic embedding (and not just a continuous map preserving capacities) rules out the case that \(\varphi'(0)\) is anti-symplectic.

 If \(n\) is odd, then any anti-symplectomorphism is orientation-reversing. Since each \(\varphi_k\) preserves the orientation, \(\varphi\) does too, implying that \(\varphi'(0)\) must be an orientation-preserving linear map. So, if \(n\) is odd, then \(\varphi'(0)\) can only be symplectic. Suppose now that \(n\) is even. On a sufficiently small ball contained in \(\mathbb{R}^2 \times B(\rho)\), the sequence \((\mathrm{id}_{\mathbb{R}^2} \times \varphi_k)\) converges uniformly to \(\mathrm{id}_{\mathbb{R}^2} \times \varphi\). Since each \(\mathrm{id}_{\mathbb{R}^2} \times \varphi_k\) is symplectic, the previous argument applies to show that \(\mathrm{id}_{\mathbb{R}^2} \times \varphi'(0)\) is symplectic, and so \(\varphi'(0)\) must be symplectic.
\end{proof}

\section{Non-squeezing for \texorpdfstring{\(\omega^k\)}{omega k}} \label{sec:powers}

In this section we discuss analogues of the non-squeezing theorem for powers of the standard symplectic form on \(\mathbb{R}^{2n}\).

\subsection{Affine non-squeezing for \texorpdfstring{\(\omega^k\)}{omega k}}
\label{section:3.1}

Let us again consider \(\mathbb{R}^{2n}\) with the standard symplectic form, which we write in terms of the standard symplectic basis denoted by \(e_1, \dots, e_n, f_1, \dots, f_n\). In~\cite{barron}, Barron and Shafiee observed that the powers \(\omega, \omega^2, \dots, \omega^n\) each satisfy an analogue of the affine non-squeezing theorem. As before, let \(B(r)\) denote the closed ball of radius \(r\) in \(\mathbb{R}^{2n}\). If \(1 \leq k \leq n\), then we define the symplectic \(k\)-cylinder \(Z_k(R)\) by 
\[
 Z_k(R) := \left\{ (x_1, \dots, x_n, y_1, \dots, y_n) \in \mathbb{R}^{2n} : \sum_{j=1}^k x_j^2 + y_j^2 \leq R^2 \right\}.
\]
The following theorem is the affine non-squeezing theorem obtained by Barron and Shafiee.

\begin{theorem} \label{thm:barron}
 ~\cite[Theorem 2.16]{barron} Suppose that \(1 \leq k \leq n\). Let \(\Psi \colon \mathbb{R}^{2n} \to \mathbb{R}^{2n}\) be an affine map of the form \(z \mapsto Az + b\), where \(A\) is a linear map preserving \(\omega^k\) and \(b \in \mathbb{R}^{2n}\). If \(\Psi(B(r)) \subseteq Z_k(R)\), then \(r \leq \frac{\sqrt[2k]{(2k)!}}{\sqrt{2}}R\).
\end{theorem}
For brevity, let us denote the coefficient of \(R\) in the conclusion of Theorem~\ref{thm:barron} by \(C_k\). That is,
\begin{equation}
 \label{eq:3.0}
 C_k = \frac{\sqrt[2k]{(2k)!}}{\sqrt{2}}.
\end{equation}
Note that \(C_k = 1\) when \(k = 1\), so we recover the classical affine non-squeezing theorem (Theorem~\ref{thm:affns}) as a special case. 

Note that when \(k > 1\), the coefficient \(C_k\) is always greater than \(1\). Also, if \(k = n\), then \(\Psi\) is volume-preserving and maps \(B(r)\) into the \emph{ball} \(Z_n(R) = B(R)\), so we obtain \(r \leq R\) trivially. So, in particular, the constant \(C_n\) is not optimal in the case that \(k = n\). In fact, we will prove later in this section (Theorem~\ref{thm:poweraffns}) that we can improve Theorem~\ref{thm:barron} to take \(C_k = 1\) for each \(k\). 

The proof of Theorem~\ref{thm:barron} given in~\cite{barron} follows essentially the same steps as the classical affine non-squeezing theorem. The first step is to argue that \(A^T\) also preserves \(\omega^k\). While this was easy to prove in the classical case, it is perhaps not so obvious here, and no proof is given in~\cite{barron}. We give a proof in the next section. In fact, we prove that Theorem~\ref{thm:barron} essentially reduces to the classical case.

Let us proceed with the proof for now under the assumption that \(A^T\) also preserves \(\omega^k\). It follows that
\begin{equation}
 \label{eq:3.1}
 k! = |\omega^k(A^T e_1, A^T f_1, \dots, A^T e_k, A^T f_k)|.
\end{equation}
In~\cite[Lemma 2.4]{barron} it is claimed that for \(u_1, \dots, u_{2k} \in \mathbb{R}^{2n}\),
\[
 \omega^k(u_1, \dots, u_{2k}) = \frac{k!}{2^k} \sum_{\sigma \in S_{2k}} \mathrm{sign}(\sigma) \prod_{i=1}^k \omega(u_{\sigma(2i-1)}, u_{\sigma(2i)}),
\]
but this is incorrect. We give the correct formula for \(\omega^k\) and then continue along the lines of~\cite{barron}.

\begin{proposition}
 For any \(u_1, \dots, u_{2k} \in \mathbb{R}^{2n}\),
 \begin{equation}
 \label{eq:3.2}
 \omega^k(u_1, \dots, u_{2k}) = \frac{1}{2^k} \sum_{\sigma \in S_{2k}} \mathrm{sign}(\sigma) \prod_{i=1}^k \omega(u_{\sigma(2i-1)}, u_{\sigma(2i)}).
 \end{equation}
\end{proposition}
\begin{proof}
 Using the definition of the wedge product in terms of the projection \(\mathrm{Alt}\), we see that 
 \begin{align*}
 \omega^k(u_1, \dots, u_{2k}) &= \frac{(2 + \cdots + 2)!}{2! \cdots 2!} \mathrm{Alt}(\omega \otimes \cdots \otimes \omega)(u_1, \dots, u_{2k}) \\
 &= \frac{(2k)!}{2^k} \cdot \frac{1}{(2k)!} \sum_{\sigma\in S_{2k}} \mathrm{sign}(\sigma)(\omega \otimes \cdots \otimes \omega)(u_{\sigma(1)}, \dots, u_{\sigma(2k)}) \\
 &= \frac{1}{2^k} \sum_{\sigma \in S_{2k}} \mathrm{sign}(\sigma) \prod_{i=1}^k \omega(u_{\sigma(2i-1)}, u_{\sigma(2i)}). \qedhere
 \end{align*}
\end{proof}

For brevity, we denote the vectors \(A^T e_1, A^T f_1, \dots, A^T e_k, A^T f_k\) by \(w_1, \dots, w_{2k}\). By using~\eqref{eq:3.2} to evaluate the right-hand side of~\eqref{eq:3.1}, we obtain
\begin{align*}
 k! &= |\omega^k(A^T e_1, A^T f_1, \dots, A^T e_k, A^T f_k)| \\
 &\leq \frac{1}{2^k}\sum_{\sigma \in S_{2k}}|\mathrm{sign}(\sigma)|\prod_{i=1}^k | \omega(w_{\sigma(2i-1)}, w_{\sigma(2i)}) | \\
 &\leq \frac{1}{2^k} \sum_{\sigma \in S_{2k}} \prod_{i=1}^k |w_{\sigma(2i-1)}||w_{\sigma(2i)}| \\
 &= \frac{1}{2^k} (2k)! \cdot |w_1| \cdots |w_{2k}|.
\end{align*}
By~\eqref{eq:3.0}, this gives 
\[
 \frac{k!}{C_k^{2k}} = \frac{k! \cdot 2^k }{(2k)!} \leq |w_1| \cdots |w_{2k}|.
\]
It follows that one of \(w_1, \dots, w_{2k}\) must have length at least \(\sqrt[2k]{k!}/C_k\). We may assume without loss of generality that \(w_1\) does. The proof concludes exactly as it does in the classical case by consideration of the image of \(\pm r \cdot w_1/|w_1|\) for an appropriate choice of sign. We note that with the correct formula~\eqref{eq:3.2} for \(\omega^k\), we actually obtain in Theorem~\ref{thm:barron} the stronger bound \(r \leq (C_k / \sqrt[2k]{k!})R\). This is still not optimal, though, since \(C_k> \sqrt[2k]{k!}\) whenever \(k > 1\).

In what follows, we show that is possible to avoid the use of~\eqref{eq:3.2} completely and to obtain the result of Theorem~\ref{thm:barron} with \(1\) in place of the coefficient \(C_k\). We will achieve this with the use of the classical \emph{Wirtinger inequality.}

\begin{theorem}[Wirtinger inequality] \label{thm:wirt}
 Suppose that \(1 \leq k \leq n\). For any orthonormal \(v_1, \dots, v_{2k}\) in \(\mathbb{R}^{2n}\), 
 \[
 |\omega^k( v_1, \dots, v_{2k} )| \leq k!,
 \]
 with equality if and only if the span of \( v_1, \dots, v_{2k} \) is invariant under \(J\).
\end{theorem}
\begin{proof}
 (Following~\cite[1.8.2]{federer}.) Let us consider first the case that \(k = 1\). By Cauchy-Schwarz, we have that 
 \[
 |\omega(v_1, v_2)| = |g(Jv_1, v_2)| \leq |Jv_1| \cdot |v_2| \leq 1,
 \]
 with equality if and only if \(Jv_1\) and \(v_2\) are colinear. That is, if and only if the span of \( v_1, v_2 \) is invariant under \(J\). 
 
 Now, suppose that \(k > 1\). Let \(S\) be the span of \(\ v_1, \dots, v_{2k} \). By the normal form theorem for skew-symmetric bilinear forms, there is an orthonormal basis \(u_1, \dots, u_{2k}\) of \(S\) such that
 \[
 \omega|_S = \sum_{j=1}^k \omega(u_{2j-1}, u_{2j}) \, u^{2j-1} \wedge u^{2j}.
 \]
 In particular,
 \[
 \omega^k(u_1, \dots, u_{2k}) = k! \cdot \prod_{j=1}^k \omega(u_{2j-1}, u_{2j}).
 \]
 Note that \(\omega^k(u_1, \dots, u_{2k})\) only depends on \(u_1 \wedge \cdots \wedge u_{2k}\). Since \(v_1, \dots, v_{2k}\) and \(u_1, \dots, u_{2k}\) are both orthonormal bases of \(S\), we have that \(v_1 \wedge \cdots \wedge v_{2k} = \varepsilon\cdot u_1 \wedge \cdots \wedge u_{2k}\) for some \(\varepsilon \in \{1, -1\}\), and it follows that
 \[
 |\omega^k(v_1, \dots, v_{2k})| = |\omega^k(u_1, \dots, u_{2k})| = k! \cdot \prod_{j=1}^k |\omega(u_{2j-1}, u_{2j})|.
 \]
 It follows from the \(k = 1\) case that \(|\omega^k(v_1, \dots, v_{2k})| \leq k!\). In particular, if equality holds, then \(|\omega(u_{2j-1}, u_{2j})| = 1\) for each \(j\), so by the \(k = 1\) case again it follows that \(S\) is invariant under \(J\). Conversely, if \(S\) is invariant under \(J\), then, for each \(j\), 
 \[ 
 Ju_{2j-1} = \sum_{\ell=1}^{2k} g(Ju_{2j-1}, u_\ell) u_\ell = \sum_{\ell=1}^{2k} \omega(u_{2j-1}, u_\ell) u_\ell = \omega(u_{2j-1}, u_{2j}) u_{2j}.
 \] 
 It follows that \(|\omega(u_{2j-1}, u_{2j})| = |Ju_{2j-1}| = |u_{2j-1}| = 1\) for each \(j\), and we conclude from the above that \(|\omega^k(v_1, \dots, v_{2k})| = k!\).
\end{proof}

We would like to apply the Wirtinger inequality to estimate the right-hand side of~\eqref{eq:3.1}. Since the Wirtinger inequality is for \emph{orthonormal} vectors, we apply the Gram-Schmidt process to \(w_1, \dots, w_{2k}\) to obtain orthogonal vectors \(v_1, \dots, v_{2k}\):
\[
 v_1 = w_1, \qquad v_j = w_j - \sum_{i=1}^{j-1} \mathrm{proj}_{v_i}(w_j).
\]
Note that \(|v_j| \leq |w_j|\) for each \(j\), for \(|w_j|^2\) is equal to a sum of squares which includes \(|v_j|^2\). Moreover, by construction we have that \(w_1 \wedge \cdots \wedge w_{2k} = v_1 \wedge \cdots \wedge v_{2k}\). It follows that 
\begin{align*}
 k! = \omega^k(w_1, \dots, w_{2k}) &= \omega^k(v_1, \dots, v_{2k}) \\
 &= |v_1| \cdots |v_{2k}| \cdot \omega\left(\frac{v_1}{|v_1|}, \dots, \frac{v_{2k}}{|v_{2k}|}\right) \\
 &\leq |v_1| \cdots |v_{2k}| \cdot k! \\
 &\leq |w_1| \cdots |w_{2k}| \cdot k!.
\end{align*}
So, \(|w_1| \cdots |w_{2k}| \geq 1\). As before, we may assume without loss of generality that \(|w_1| \geq 1\), and then the proof that \(r \leq R\) follows exactly as it does in the classical case. So, modulo the proof that the stabilizer of \(\omega^k\) is closed under transposition (which we prove in the next section), we have established the following improved affine non-squeezing theorem for \(\omega^k\).

\begin{theorem}[Improved affine non-squeezing theorem for \(\omega^k\)] \label{thm:poweraffns}
 Let \(1 \leq k \leq n\). Let \(\Psi \colon \mathbb{R}^{2n} \to \mathbb{R}^{2n}\) be an affine map of the form \(z \mapsto Az + b\), where \(A\) is a linear map preserving \(\omega^k\) and \(b \in \mathbb{R}^{2n}\). If \(\Psi(B(r)) \subseteq Z_k(R)\), then \(r \leq R\).
\end{theorem}

\subsection{The stabilizer of \texorpdfstring{\(\omega^k\)}{omega k}}
\label{section:3.2}

The proofs of Theorem~\ref{thm:barron} and Theorem~\ref{thm:poweraffns} depended on the as-of-then unproven fact that the stabilizer of \(\omega^k\) is closed under transposition. In this section, we prove that this is the case. In fact, we prove the following stronger result.

\begin{theorem} \label{thm:power-op}
 Suppose that \(1 \leq k \leq n - 1\). Let \(A \colon \mathbb{R}^{2n} \to \mathbb{R}^{2n}\) be a linear map for which \(A^*(\omega^k) = \omega^k\). If \(k\) is odd, then \(A^*\omega = \omega\), and if \(k\) is even, then \(A^*\omega = \omega\) or \(A^*\omega = -\omega\).
\end{theorem}

\begin{remark} \label{rmk:powers}
The statement of Theorem~\ref{thm:power-op} is likely well-known to experts, but seems difficult to find stated explicitly in the literature. Indeed, in~\cite[Section 2.3]{barron} it is implied that the result of Theorem~\ref{thm:power-op} was unknown. Similarly, although a global non-linear version of Theorem~\ref{thm:power-op} is proved in~\cite[Theorem 11]{alizadeh} characterizing those diffeomorphisms of a symplectic manifold $(M, \omega)$ which are connected to the identity and preserve $\omega^k$ as being themselves symplectomorphisms, no explicit mention of the corresponding linear statement is given in~\cite{alizadeh}. Although the proof in~\cite{alizadeh} also uses the Lefschetz decomposition as we do, it employs smooth methods rather than the induction argument we give, and as such says nothing about those diffeomorphisms preserving $\omega^k$ which are not connected to the identity. We use our Theorem~\ref{thm:power-op} to give a simple proof of~\cite[Theorem 11]{alizadeh} in Theorem~\ref{thm:alizadeh} below. An infinitesimal version of~\cite[Theorem 11]{alizadeh} is also proved in~\cite[Lemma 5.5]{RW}. (See also~\cite{MW}.)
\end{remark}

Reformulating the statement of the theorem, if we let \(\mathrm{Sp}_k(2n, \mathbb{R})\) denote the stabilizer of \(\omega^k\) in \(\mathrm{GL}(2n, \mathbb{R})\), then we have that
\[
 \mathrm{Sp}_k(2n, \mathbb{R}) = \begin{cases}
 \mathrm{Sp}(2n, \mathbb{R}), & \text{if \(k\) is odd and \(k \leq n - 1\),} \\
 \mathrm{Sp}(2n, \mathbb{R}) \cup \mathrm{Sp}^{-}(2n, \mathbb{R}), & \text{if \(k\) is even and \(k \leq n - 1\),} \\
 \mathrm{SL}(2n, \mathbb{R}), & \text{if \(k = n\),}
 \end{cases}
\]
where 
\[
 \mathrm{Sp}^-(2n, \mathbb{R}) = \{ A \in \mathrm{GL}(2n, \mathbb{R}) : A^*\omega = -\omega \}.
\]
Actually, we will prove the following stronger fact which implies Theorem~\ref{thm:power-op}.

\begin{theorem} \label{thm:power}
 Suppose that \(1 \leq k \leq n - 1\). Let \(\omega = \sum_{i=1}^n e^i \wedge f^i\) be the standard symplectic form on \(\mathbb{R}^{2n}\) and let \(\Omega\) be any \(2\)-form on \(\mathbb{R}^{2n}\) with constant coefficients such that \(c\Omega^k = \omega^k\) for some (necessarily non-zero) constant \(c\). Then, there exist scalars \(\lambda_1, \dots, \lambda_n\) such that \(\Omega = \sum_{i=1}^n \lambda_i \, e^i \wedge f^i\).
\end{theorem}
The role of the constant \(c\) will become clear in the course of the proof.

Let us first prove that Theorem~\ref{thm:power-op} follows from Theorem~\ref{thm:power}. Given that \(A^*(\omega^k) = \omega^k\), let \(\Omega = A^*\omega\). Then \(\Omega^k = \omega^k\), so that \(\Omega = \sum_{i=1}^n \lambda_i \, e^i \wedge f^i\) for some scalars \(\lambda_1, \dots, \lambda_n\) by Theorem~\ref{thm:power}. Since
\[
 \Omega^k = k! \sum_{i_1 < \cdots < i_k} \lambda_{i_1} \cdots \lambda_{i_k} \, e^{i_1} \wedge f^{i_1} \wedge \cdots \wedge e^{i_k} \wedge f^{i_k},
\]
we see from the equation \(\Omega^k = \omega^k\) that the product of any \(k\) of \(\lambda_1, \dots, \lambda_n\) without repetition in the indices must equal \(1\). Since \(k \leq n - 1\), it follows that \(\lambda_1 = \cdots = \lambda_n\). For example,
\[
 \lambda_1 = \frac{\lambda_1 \lambda_3 \cdots \lambda_{k+1}}{\lambda_3 \cdots \lambda_{k+1}} = \frac{\lambda_2 \lambda_3 \cdots \lambda_{k+1}}{\lambda_3 \cdots \lambda_{k+1}} = \lambda_2.
\]
This shows that \(\lambda_i = \lambda_j\) for all \(i\) and \(j\). (If we had \(k = n\), then the only equation we would have would be \(\lambda_1 \cdots \lambda_n = 1\), which cannot be solved uniquely for any \(\lambda_i\).) Since \(\lambda_1^k = 1\), it follows that \(\lambda_1 = \cdots = \lambda_n = 1\) if \(k\) is odd and \(\pm 1\) if \(k\) is even.

To prove Theorem~\ref{thm:power}, we need the following lemma.
\begin{lemma} \label{lem:lefschetz}
 Let \((V, \omega)\) be a \(2n\)-dimensional real symplectic vector space. For each \(j = 1, \dots, n\), the map
 \[
 L^j \colon \Lambda^{n-j} V^* \to \Lambda^{n+j} V^*, \qquad \theta \mapsto \theta \wedge \omega^j
 \]
 is an isomorphism.
\end{lemma}

Note that it is essential for the lemma that we consider the map \(\theta \mapsto \theta \wedge \omega^j\) only on \(\Lambda^{n-j}V^*\), for \(\Lambda^{n-j}V^*\) and \(\Lambda^{n+j}V^*\) have the same dimension. To prove Theorem~\ref{thm:power}, we actually only need the statement of Lemma~\ref{lem:lefschetz} in the case that \(j = n - 1\), but the proof of the lemma for all \(j\) is nearly identical to the proof for \(j = n - 1\).
\begin{proof}[Proof of Lemma~\ref{lem:lefschetz}]
 We proceed by induction on \(j\), starting from \(j = n\) and proceeding backwards. The base case holds simply because \(L^n\) is multiplication on scalars by the volume form \(\omega^n\). Suppose that for some \(j \leq n\), the map \(L^j\) is an isomorphism. To prove that \(L^{j-1}\) is an isomorphism, it suffices to prove that it is injective, so suppose that \(\theta\) is a \((n-j+1)\)-form such that \(L^{j-1}(\theta) = 0\). It follows that
 \[
 0 = L^{j-1}(\theta) \wedge \omega = \theta \wedge \omega^j.
 \]
 We contract both sides of the above equation by an arbitrary \(v\) in \(V\) to obtain
 \[
 0 = (i_v\theta) \wedge \omega^j \pm j \theta \wedge (i_v \omega) \wedge \omega^{j-1}.
 \]
 Since \(\theta \wedge \omega^{j-1} = 0\), the second term on the right-hand side vanishes and we are left with \((i_v\theta) \wedge \omega^j = 0\). But \(i_v\theta\) is of degree \(n - j\), so this says that \(L^j(i_v\theta) = 0\). By the inductive hypothesis, \(i_v\theta = 0\), and since \(v\) was arbitrary, \(\theta = 0\).
\end{proof}

We are now ready to prove Theorem~\ref{thm:power}. In the following proof, we drop the wedge product symbols for brevity, and all forms are forms with constant coefficients.

\begin{proof}[Proof of Theorem~\ref{thm:power}]
We proceed by induction on \(n\). If \(n = 2\), then \(k = 1\) and \(c\Omega = \omega\). This implies that \(\Omega = c^{-1}\sum e^i f^i\), proving the base case. Now, suppose the theorem holds for a given \(n \geq 2\). Let \(\omega = \sum_{i=1}^{n+1} e^i f^i\) be the standard symplectic form on \(\mathbb{R}^{2(n+1)}\), and suppose that \(\Omega\) is a \(2\)-form on \(\mathbb{R}^{2(n+1)}\) such that \(c\Omega^k = \omega^k\) for some scalar \(c\) and some integer \(k\) with \(1 \leq k \leq (n + 1) - 1 = n\). If \(k = 1\), then we are immediately done as in the base case, so let us assume that \(k \geq 2\). Decompose \(\Omega\) into four components as follows:
 \[
 \Omega = \Omega_0 + \alpha + \beta + \gamma,
 \]
 where
 \begin{itemize}
 \item \(\Omega_0\) is a \(2\)-form without any \(e^{n+1}\) or \(f^{n+1}\) terms,
 
 \item \(\alpha = e^{n+1} \alpha_0\) for some \(1\)-form \(\alpha_0\) without any \(e^{n+1}\) or \(f^{n+1}\) terms,
 
 \item \(\beta = f^{n+1} \beta_0\) for some \(1\)-form \(\beta\) of the same type as \(\alpha_0\), and 

 \item \(\gamma = \lambda \, e^{n+1} f^{n+1}\) for some scalar \(\lambda\).
 \end{itemize}
 Note that \(\alpha^2 = \beta^2 = \gamma^2 = \alpha\gamma = \beta\gamma = 0\). It follows from this and the binomial theorem that
 \begin{align*}
 \Omega^k &= \Omega_0^k + k \Omega_0^{k-1}(\alpha + \beta + \gamma) + \frac{k(k-1)}{2} \Omega_0^{k-2}(\alpha + \beta + \gamma)^2 \\
 &= \Omega_0^k + k \Omega_0^{k-1}(\alpha + \beta + \gamma) + k(k - 1) \Omega_0^{k-2} \alpha\beta.
 \end{align*}
 If we similarly decompose \(\omega\) into \(\omega_0 + e^{n+1} f^{n+1}\), where \(\omega_0 = \sum_{i=1}^n e^i f^i\), then we obtain
 \[
 \omega^k = \omega_0^k + k \omega_0^{k-1} e^{n+1} f^{n+1}.
 \]
 Comparing the components of \(c\Omega^k\) and of \(\omega^k\) which have no \(e^{n+1}\) or \(f^{n+1}\), only one of the two, and both gives us the following four equations:
 \begin{align*}
 c\Omega_0^k &= \omega_0^k, \\
 \Omega_0^{k-1} \alpha_0 &= 0, \\
 \Omega_0^{k-1} \beta_0 &= 0, \\
 \lambda c\Omega_0^{k-1} - c(k - 1)\Omega_0^{k-2} \alpha_0 \beta_0 &= \omega_0^{k-1}.
 \end{align*}
 It should be emphasized that all of the forms in the above four equations only contain \(e^1, f^1, \dots, e^n, f^n\), and we think of them as forms on \(\mathbb{R}^{2n} = \mathrm{span}\{e_1, f_1, \dots, e_n, f_n\}\). 

 We first prove that \(\Omega_0\) is symplectic. Suppose that \(i_v\Omega_0 = 0\) for some \(v\) in \(\mathbb{R}^{2n}\). It follows that
 \[
 0 = i_v(c\Omega_0^k) = i_v(\omega_0^k) = k (i_v\omega_0) \omega_0^{k-1}.
 \]
 This implies that \((i_v\omega_0) \omega_0^{n-1} = 0\), and then it follows from Lemma~\ref{lem:lefschetz} that \(i_v\omega_0 = 0\). Since \(\omega_0\) is the standard symplectic form on \(\mathbb{R}^{2n}\), this forces \(v\) to be zero, hence \(\Omega_0\) is symplectic. From the equation \(\Omega_0^{k-1} \alpha_0 = 0\) it follows that \(\Omega_0^{n-1} \alpha_0 = 0\), and since \(\Omega_0\) is symplectic, Lemma~\ref{lem:lefschetz} implies that \(\alpha_0 = 0\). The same argument gives \(\beta_0 = 0\), which leaves us with the fourth equation,
 \[
 \lambda c \Omega_0^{k-1} = \omega_0^{k-1}.
 \]
 By the inductive hypothesis, \(\Omega_0 = \sum_{i=1}^n \lambda_i \, e^i f^i\) for some scalars \(\lambda_1, \dots, \lambda_n\), and we conclude that
 \[
 \Omega = \Omega_0 + \lambda \, e^{n+1} f^{n+1} = \lambda_1 \, e^1 f^1 + \cdots + \lambda_n \, e^n f^n + \lambda \, e^{n+1} f^{n+1}. \qedhere
 \]
\end{proof}

Note that if we had not included the coefficient \(c\) in the statement of Theorem~\ref{thm:power}, then we would have ended up with the equation \(\lambda \Omega_0^{k-1} = \omega_0^{k-1}\), to which the inductive hypothesis would not apply unless we were able to pull \(\lambda\) into the power of \(k - 1\).

\begin{proof}[Alternate proof of Theorem~\ref{thm:poweraffns}]
Using Theorem~\ref{thm:power-op}, we can give an alternate proof of Theorem~\ref{thm:poweraffns} by reducing it to the symplectic case. Suppose that \(\psi\) is an affine map which preserves \(\omega^k\) and takes \(B(r)\) into \(Z_k(R)\). We have already seen that if \(k = n\), then \(r \leq R\) due to the fact that \(\psi\) must be volume-preserving and \(Z_n(R) = B(R)\). Suppose that \(k < n\). In this case, it follows from Theorem~\ref{thm:power-op} that \(\psi\) must preserve \(\omega\) up to a sign. Since \(Z_k(R) \subseteq Z(R)\), we have that \(\psi\) is an affine symplectomorphism or anti-symplectomorphism taking \(B(r)\) into \(Z(R)\). By the classical affine non-squeezing theorem, we conclude that \(r \leq R\).
\end{proof}

In Section~\ref{section:2.2}, after we proved the affine non-squeezing theorem for \(\omega\), we showed that we can characterize symplectic and anti-symplectic linear maps as those maps preserving the linear symplectic widths of ellipsoids. Now that we have proven the affine non-squeezing theorem for \(\omega^k\), it is natural to try to do the same for maps preserving \(\omega^k\). As we will see below in Proposition~\ref{prop:kwidth}, it is a consequence of Theorem~\ref{thm:power-op} that the natural analogue of the linear symplectic width for \(\omega^k\) does not result in a new invariant of subsets of \(\mathbb{R}^{2n}\). In particular, the natural \(k\)-symplectic analogue of Theorem~\ref{thm:linchar} will reduce to the symplectic case when \(k \leq n - 1\).

We define the \emph{linear \(k\)-symplectic width} of a subset \(A\) of \(\mathbb{R}^{2n}\) by 
\[
 w_k(A) := \sup\{ \mathcal{V}_{2k} r^{2k} : \Psi(B(r)) \subseteq A \text{ for some affine map } \Psi \text{ preserving } \omega^k\},
\]
where \(\mathcal{V}_{2k}\) is the volume of the unit ball in \(\mathbb{R}^{2k}\). Note that the linear symplectic width \(w_L(A)\) is exactly \(w_1(A)\); more generally, we have the following expression relating the two.

\begin{proposition} \label{prop:kwidth}
 Suppose \(1 \leq k \leq n - 1\). For any subset \(A\) of \(\mathbb{R}^{2n}\), 
 \[
 w_k(A) = \frac{\mathcal{V}_{2k}}{\pi^k} \cdot [w_L(A)]^k.
 \]
\end{proposition}
\begin{proof}
 If \(B(r)\) embeds into \(A\) by an \(\omega\)-preserving affine map, then it certainly embeds into \(A\) by an \(\omega^k\)-preserving affine map, so
 \[
 w_k(A) \geq \mathcal{V}_{2k} r^{2k} = \frac{\mathcal{V}_{2k}}{\pi^k} \cdot [\pi r^2]^k.
 \]
 It follows that \(w_k(A) \geq \frac{\mathcal{V}_{2k}}{\pi^k} \cdot [w_L(A)]^k\). Conversely, if \(B(r)\) embeds into \(A\) by an \(\omega^k\)-preserving affine map, then that affine map is either symplectic or anti-symplectic by Theorem~\ref{thm:power-op}, so that (recalling Proposition~\ref{prop:antisymp}), we have 
 \[
 \frac{\mathcal{V}_{2k}}{\pi^k} \cdot [w_L(A)]^k \geq \frac{\mathcal{V}_{2k}}{\pi^k} \cdot [\pi r^2]^k = \mathcal{V}_{2k} r^{2k}.
 \]
 It follows that \(\frac{\mathcal{V}_{2k}}{\pi^k} \cdot [w_L(A)]^k \geq w_k(A)\), so in fact equality must hold.
\end{proof}

\subsection{Applications to non-linear \texorpdfstring{\(\omega^k\)}{omega k}-preserving maps}
\label{section:3.3}

We turn to studying \(k\)-symplectic analogues of the results of Section~\ref{section:2.4}. We begin by noting that the application of Theorem~\ref{thm:power-op} used to prove the affine non-squeezing theorem for \(\omega^k\) by reduction to the symplectic case readily generalizes to prove the following \emph{non-linear} non-squeezing theorem for \(\omega^k\).

\begin{proposition}
 Suppose \(1 \leq k \leq n\). If there exists an \(\omega^k\)-preserving embedding of \(B(r)\) into \(Z_k(R)\), then \(r \leq R\).
\end{proposition}
\begin{proof}
 Suppose \(k \leq n - 1\). Any \(\omega^k\)-preserving embedding is either symplectic or anti-symplectic by Theorem~\ref{thm:power-op} applied pointwise to the derivatives of the map. Since \(Z_k(R) \subseteq Z(R)\), it follows from Gromov's non-squeezing theorem (Theorem~\ref{thm:gns}) that \(r \leq R\). For \(k = n\), the assumption is that \(B(r)\) admits a volume-preserving embedding into \(Z_n(R) = B(R)\), so the conclusion \(r \leq R\) is clear.
\end{proof}

Just as we did at the end of the previous section, we define the non-linear \(k\)-symplectic width \(w_k(A)\) of a subset of \(\mathbb{R}^{2n}\) by
\[
 w_k(A) := \sup\{\mathcal{V}_{2k}r^{2k} : \psi(B(r)) \subseteq A \text{ for some \(\omega^k\)-preserving embedding } \psi \colon B(r) \to \mathbb{R}^{2n}\}.
\]
The same formula as in Proposition~\ref{prop:kwidth} that related the linear \(k\)-symplectic width to the linear symplectic width holds for the non-linear \(k\)-symplectic width and the Gromov width, by the same proof. We can also define the ``maximal \(k\)-capacity'' 
\[
 \overline{w_k}(A) := \inf\{\mathcal{V}_{2k} R^{2k} : \psi(A) \subseteq Z_k(R) \text{ for some \(\omega^k\)-preserving embedding } \psi \colon A \to \mathbb{R}^{2n}\}.
\]
The maximal \(k\)-capacity has the following relation to the classical maximal capacity \(\overline{w}\).

\begin{proposition} \label{prop:maxcapacity}
 Suppose \(1 \leq k \leq n - 1\). For any subset \(A\) of \(\mathbb{R}^{2n}\), we have
 \[
  \overline{w}_k(A) \geq \frac{\mathcal{V}_{2k}}{\pi^k} \cdot [\overline{w}(A)]^k.
 \]
\end{proposition}
\begin{proof}
 If \(A\) embeds into \(Z_k(R)\) by an \(\omega^k\)-preserving embedding, then it embeds into \(Z(R)\) by such an embedding since \(Z_k(R) \subseteq Z(R)\). Any \(\omega^k\)-preserving map is pointwise symplectic or antisymplectic by Theorem~\ref{thm:power-op}, so it follows that 
 \[
    \frac{\mathcal{V}_{2k}}{\pi^k} \cdot [\overline{w}(A)]^k \leq \frac{\mathcal{V}_{2k}}{\pi^k} \cdot [\pi R^2]^k = \mathcal{V}_{2k} R^{2k}.
 \]
 Taking infimums gives the desired inequality.
\end{proof}
It is not clear whether we can establish equality in general. Any \(\omega\)-preserving embedding into \(Z(R)\) is certainly an \(\omega^k\)-preserving embedding into \(Z(R)\), but it is not clear how to go from this to an embedding into the \(k\)-cylinder \(Z_k(R)\).


The proof of Theorem~\ref{thm:nonlinchar} relied on Theorem~\ref{thm:limit}, whose proof required showing that a particular linear map preserved the maximal capacity of every ellipsoid and the fact that every normalized capacity takes the same value on any given ellipsoid (Theorem~\ref{thm:ellcap}). It is not clear whether this holds for the maximal \(k\)-capacity \(\overline{w_k}\).

On the other hand, by using Theorem~\ref{thm:power-op}, it is easy to establish (by reduction to the symplectic case) a \(k\)-symplectic analogue of Theorem~\ref{thm:c0closure}. Given a symplectic manifold \((M, \omega)\), we let \(\mathrm{Symp}(M, \omega)\) denote the group of self-diffeomorphisms of \(M\) preserving \(\omega\), and we also let \(\mathrm{Symp}_k(M, \omega)\) denote the group of self-diffeomorphisms preserving \(\omega^k\).

\begin{proposition}
 For any symplectic \(2n\)-manifold \((M, \omega)\) and any integer \(k\) with \(1 \leq k \leq n - 1\), the group \(\mathrm{Symp}_k(M, \omega)\) is closed in \(\mathrm{Diff}(M)\) with respect to the \(C^0\)-compact-open topology.
\end{proposition}
\begin{proof}
 If \(k\) is odd, then we clearly have that \(\mathrm{Symp}_k(M, \omega) = \mathrm{Symp}(M, \omega)\) by Darboux's theorem and Theorem~\ref{thm:power-op}, so the result reduces to Theorem~\ref{thm:c0closure} in this case. 
 
 Let us assume that \(k\) is even. Just as in the proof of the classical version of the theorem, we may assume by Darboux's theorem that we have a sequence \((\varphi_j)\) of \(\omega^k\)-preserving embeddings of the ball \(B(\rho)\) into \(\mathbb{R}^{2n}\) which converges uniformly to a smooth embedding \(\varphi \colon B(\rho) \to \mathbb{R}^{2n}\) and that we need to prove that \(\varphi'(0)\) preserves \(\omega^k\). By Theorem~\ref{thm:power-op}, each \(\varphi_j\) is either symplectic or anti-symplectic. In particular, each \(\varphi_j\) preserves the Gromov width of every ellipsoid centered at the origin in \(B(\rho)\). Therefore, by Theorem~\ref{thm:limit}, it follows that \(\varphi'(0)\) is symplectic or anti-symplectic. Since \(k\) is even, it follows that \(\varphi'(0)\) preserves \(\omega^k\).
\end{proof}

If \(M\) is not connected and if \(k\) is even, then an \(\omega^k\)-preserving map could be symplectic on one component and anti-symplectic on another, so the description of \(\mathrm{Symp}_k(M, \omega)\) in general is more complicated. But when \(M\) is connected, it is easy to describe the groups \(\mathrm{Symp}_k(M, \omega)\) for \(1 \leq k \leq n - 1\).

\begin{theorem} \label{thm:powermfd}
 If \((M, \omega)\) is a connected symplectic \(2n\)-manifold, then
 \[
 \mathrm{Symp}_k(M, \omega) = 
 \begin{cases}
 \mathrm{Symp}(M, \omega), & \text{if \(k\) is odd and \(k \leq n - 1\),} \\
 \mathrm{Symp}(M, \omega) \cup \mathrm{Symp}^-(M, \omega), & \text{if \(k\) is even and \(k \leq n - 1\),}
 \end{cases}
 \]
 where \(\mathrm{Symp}^-(M, \omega)\) is the group of anti-symplectomorphisms.
\end{theorem}
\begin{proof}
 Apply the Darboux theorem and then Theorem~\ref{thm:power-op}.
\end{proof}

We conclude this section with an application to some recent results found in~\cite{alizadeh} on the groups \(\mathrm{Symp}_k(M, \omega)\). By using Theorem~\ref{thm:powermfd}, we are able to give a quick proof of the main theorem of~\cite{alizadeh}.

\begin{theorem}[{\cite[Theorem 11]{alizadeh}}] \label{thm:alizadeh}
 For any symplectic \(2n\)-manifold \((M, \omega)\) and any integer \(k\) with \(1 \leq k \leq n - 1\), the identity component of \(\mathrm{Symp}_k(M, \omega)\) is equal to the identity component of \(\mathrm{Symp}(M, \omega)\).
\end{theorem}
\begin{proof}
 Suppose that \(f\) is an \(\omega^k\)-preserving self-diffeomorphism of \(M\) that is connected to the identity by a path of such maps. Then \(f\) restricts to a self-diffeomorphism of each component of \(M\) which preserves \(\omega^k\), so, by Theorem~\ref{thm:powermfd}, we see that \(f\) is either symplectic or anti-symplectic on each component. Since \(f\) is connected to the identity, which is a symplectomorphism, it must be a symplectomorphism itself.
\end{proof}

\section{(Affine) non-squeezing for general calibrations} \label{sec:calibrations}

We aim to generalize the affine non-squeezing theorem for the symplectic form \(\omega\) and its powers \(\omega^k\) to a more general class of forms known as \emph{calibrations}. In this section, we define \emph{calibrations} and discuss when analogues of the affine non-squeezing theorems (Theorems~\ref{thm:affns} and~\ref{thm:poweraffns}) hold for them.

\subsection{When does a calibration satisfy the non-squeezing theorem?}
\label{section:4.1}

We are interested in developing invariants for \emph{calibrated geometries} which are analogous to invariants from symplectic geometry. So, it is natural to investigate how much of the theory of Sections~\ref{sec:classical} and~\ref{sec:powers} can be generalized; in particular, whether the non-squeezing theorems generalize and whether calibrated geometries support an interesting theory of capacities. We begin by exploring this question in the affine case, as we did in Sections~\ref{sec:classical} and~\ref{sec:powers}.

The proof of the affine non-squeezing theorem for \(\omega^k\) (Theorem~\ref{thm:poweraffns}) essentially depended on the following two facts about the form \(\omega^k\):
\begin{enumerate}
 \item The group of linear maps preserving \(\omega^k\) is closed under transposition (which follows from Theorem~\ref{thm:power-op}).

 \item The Wirtinger inequality (Theorem~\ref{thm:wirt}). (In the proof of the affine non-squeezing theorem for \(\omega\), Theorem~\ref{thm:affns}, we proved Wirtinger's inequality in the case \(k = 1\).)
\end{enumerate}
The Wirtinger inequality, in more general terms, states that the form \(\omega^k / k!\) is a \emph{calibration} on \(\mathbb{R}^{2n}\) whose \emph{calibrated subspaces} are exactly the \(2k\)-dimensional complex subspaces. 

Let us review the notion of a calibration and some of the basic theory. We say that a linear \(k\)-form \(\alpha\) on an inner product space \((V, g)\) is a \emph{calibration} if it has \emph{comass one}. That is, if 
\begin{equation}
 \label{eq:4.1}
 |\alpha(v_1, \dots, v_k)| \leq 1 \qquad \text{for all orthonormal } v_1, \dots, v_k \in V,
\end{equation}
and if equality holds for at least one \(k\)-tuple \((v_1, \dots, v_k)\). We say that a \(k\)-dimensional \emph{oriented} subspace \(L\) of \(V\) is \(\alpha\)-\emph{calibrated} (or just \emph{calibrated} if the form is clear from the context) if there is an oriented orthonormal basis \(v_1, \dots, v_k\) of \(L\) such that \(\alpha(v_1, \dots, v_k) = 1\). Note that if this is true for one such basis, then it is true for every such basis; more succinctly, \(\alpha|_L = \mathrm{vol}_L\). The application of Gram-Schmidt used to deduce Theorem~\ref{thm:poweraffns} from the Wirtinger inequality generalizes to show that, in fact, \(\alpha\) is a calibration if and only if
\[
 |\alpha(v_1, \dots, v_k)| \leq |v_1| \cdots |v_k| \qquad \text{for all } v_1, \dots, v_k \in V,
\]
with equality for at least one \((v_1, \dots, v_k)\).

More generally, calibrations can be defined on Riemannian manifolds. A \emph{calibration} on a Riemannian manifold \((M, g)\) is defined to be a \emph{closed} \(k\)-form \(\alpha\) such that for each \(x\) in \(M\), the linear \(k\)-form \(\alpha_x\) on \(T_x M\) has comass one. A \(k\)-dimensional \emph{oriented} submanifold \(L\) of \(M\) is said to be \emph{\(\alpha\)-calibrated} if \(T_x L\) is \(\alpha_x\)-calibrated for each \(x\) in \(L\). 

We will not make use of this, but one motivation for studying calibrations comes from the fact that calibrated submanifolds are minimal submanifolds (\cite[Theorem 4.2]{harlaw}), and in fact that are locally volume-minimizing. The equations for a calibrated submanifold are first-order fully nonlinear, whereas the minimal submanifold equations are second-order quasilinear.

Before trying to generalize the affine non-squeezing theorem, we discuss some examples of calibrations. Most of our examples come from the theory of Riemannian manifolds with special holonomy.
\begin{enumerate}
 \item As we have seen, the Wirtinger inequality implies that each \(\omega^k/k!\) on \(\mathbb{R}^{2n}\) for \(k = 1, \dots, n\) is a calibration whose calibrated subspaces are the \(2k\)-dimensional complex subspaces. It follows that the calibrated \emph{submanifolds} are the complex \(2k\)-dimensional submanifolds. (This corresponds to K\"ahler geometry.)

 \item The \(3\)-form $\varphi$ on \(\mathbb{R}^7\) associated to \(\mathrm{G}_2\)-geometry, as well as its Hodge dual \(4\)-form $\psi$, are both calibrations on \(\mathbb{R}^7\) whose calibrated subspaces are called \emph{associative} and \emph{coassociative}, respectively. (We refer to~\cite[Definition 3.14]{spiro-g2} for the definitions of these two forms.) We will return to this example in Section~\ref{section:4.2}. Similarly, the fundamental \(4\)-form \(\Phi\) on \(\mathbb{R}^8\) associated to \(\mathrm{Spin}(7)\)-geometry is a calibration whose calibrated subspaces are called \emph{Cayley}.

 \item Consider the holomorphic volume form \(\Omega = dz^1 \wedge \cdots \wedge dz^n\) on \(\mathbb{C}^n = \mathbb{R}^{2n}\). The real part of \(\Omega\) is a calibration whose calibrated submanifolds are said to be \emph{special Lagrangian submanifolds}. (More generally, \(\mathrm{Re}(e^{i\theta}\Omega)\) is a calibration for every \(\theta \in \mathbb{R}\).) The question of (affine) non-squeezing for this calibrated geometry will be the focus of Section~\ref{sec:slag}.

 \item The form \(e^1 \wedge \cdots \wedge e^k\) on \(\mathbb{R}^n\), where \(e_1, \dots, e_n\) is any orthonormal basis with dual basis \(e^1, \dots, e^n\), is a calibration whose only calibrated subspace is the span of \(e_1, \dots, e_k\). (This is a straightforward consequence of Cauchy-Schwarz on \(\Lambda^k \mathbb{R}^n\).)
\end{enumerate}

In the remainder of this section, we focus on constant coefficient calibrations on \(\mathbb{R}^n\) and generalizing the affine non-squeezing theorem. To formulate the non-squeezing theorem for a given calibration, we need to replace the symplectic cylinder with a ``calibrated cylinder'' which we now define. For a given calibration \(\alpha\) on \(\mathbb{R}^n\) and a given \(\alpha\)-calibrated subspace \(L\), we define the corresponding calibrated cylinder of radius \(R\) by
\begin{equation} \label{eq:calib-cylinder}
 Z_L(R) = \{x \in \mathbb{R}^n : | \mathrm{proj}_L(x) |^2 \leq R^2\}.
\end{equation}
The symplectic \(k\)-cylinder \(Z_k(R)\) is the calibrated cylinder for the calibration \(\omega^k/k!\) with respect to the calibrated subspace spanned by \(e_1, f_1, \dots, e_k, f_k\).

\begin{remark} \label{rmk:cylinders}
 The adjective ``calibrated'' in ``calibrated cylinder'' refers to the given calibrated subspace \(L\) of \(\mathbb{R}^n\) on which the cylinder is based. We emphasize that a calibrated cylinder in \(\mathbb{R}^n\) with respect to the calibration \(k\)-form \(\alpha\) is an \(n\)-dimensional submanifold with boundary in \(\mathbb{R}^n\) given by~\eqref{eq:calib-cylinder}. It should not be confused with an \(\alpha\)-calibrated submanifold, which is \(k\)-dimensional.   
\end{remark}

The fact that \(\omega^k/k!\) had comass one was one of the ingredients we needed to prove the affine non-squeezing theorem in Section~\ref{sec:powers}. The other was the fact that its stabilizer is closed under transposition. As we now prove, if we assume that this holds for some given calibration, then that calibration satisfies the following affine non-squeezing theorem. The proof proceeds along the exact same lines as the proofs of the affine non-squeezing theorems in Sections~\ref{sec:classical} and~\ref{sec:powers} (Theorems~\ref{thm:affns} and~\ref{thm:poweraffns}, respectively).

\begin{proposition} \label{prop:calibns}
 Let \(\alpha \in \Lambda^k(\mathbb{R}^n)^*\) be a calibration. If the stabilizer of \(\alpha\) in \(\mathrm{GL}(n, \mathbb{R})\) is closed under transposition, then \(\alpha\) has the following affine non-squeezing property: if \(\psi\) is an affine map preserving \(\alpha\) for which \(\psi(B(r)) \subset Z_L(R)\) for some \(\alpha\)-calibrated subspace \(L\), then \(r \leq R\).
\end{proposition}
\begin{proof}
 Suppose that \(\psi\) is given by \(\psi(x) = \Psi x + x_0\), where \(\Psi^*\alpha = \alpha\) and \(x_0 \in \mathbb{R}^n\). Choose an oriented orthonormal basis \(e_1, \dots, e_k\) of \(L\), so that \(\alpha(e_1, \dots, e_k) = 1\) since \(L\) is \(\alpha\)-calibrated. Since \(\Psi^T\) must preserve \(\alpha\) by assumption, it follows that 
 \[
 1 = |\alpha(e_1, \dots, e_k)| = |\alpha(\Psi^T e_1, \dots, \Psi^T e_k)| \leq |\Psi^T e_1| \cdots |\Psi^T e_k|.
 \]
 Hence one of \(|\Psi^Te_1|, \dots, |\Psi^Te_k|\) has norm at least \(1\). We may assume without loss of generality that \(|\Psi^Te_1| \geq 1\). Define \(x = \pm r \cdot \Psi^Te_1/|\Psi^Te_1|\), where the sign is to be determined. Since \(x \in B(r)\), we have \(\psi(x) \in Z_L(R)\), and so
 \begin{align*}
 R^2 &\geq \langle e_1, \psi(x) \rangle^2 + \cdots + \langle e_k, \psi(x) \rangle^2 \\
 &\geq \langle e_1, \Psi x + x_0\rangle^2 \\
 &= [\langle \Psi^T e_1, x \rangle + \langle e_1, x_0 \rangle ]^2 \\
 &= [\pm r |\Psi^T e_1| + \langle e_1, x_0 \rangle]^2.
 \end{align*}
 By choosing the sign in the definition of \(x\) to be the sign of \(\langle e_1, x_0 \rangle\), it follows that \(R^2 \geq r^2 |\Psi^Te_1|^2 \geq r^2\), and so \(R \geq r\).
\end{proof}

\begin{remark}
The condition that the stabilizer of the calibration \(\alpha\) is closed under transposition is not a necessary condition for \(\alpha\) to satisfy the non-squeezing property of Proposition~\ref{prop:calibns}. The fourth family of calibrations in our list of examples above provides us with an example. Let \(\alpha = e^1\), where \(e_1, e_2\) is an orthonormal basis of \(\mathbb{R}^2\) with dual basis \(e^1, e^2\). If \(\Psi\) is an invertible linear map which preserves \(\alpha\), then it is necessarily of the form
\[
 \Psi = \begin{pmatrix} 1 & 0 \\ c & d \end{pmatrix}, \qquad d \neq 0
\]
with respect to the basis \(e_1, e_2\). Suppose that the affine map \(\psi\) given by \(\psi(z) = \Psi z + z_0\) sends \(B(r)\) into \(Z_L(R)\), where \(z_0 = (x_0, y_0)\) and \(L = \mathrm{span}\{e_1\}\). Since the orthogonal projection onto \(L\) is just the projection onto the first component, this means that \(x + x_0 \in [-R, R]\) whenever \(x^2 + y^2 \leq r^2\). In particular, if we take \(y = 0\), then we see that \(x + x_0 \in [-R, R]\) whenever \(x \in [-r, r]\). So, \([-R, R]\) contains the interval \(x_0 + [-r, r]\) of length \(2r\), meaning that \(2R \geq 2r\). However, \(\Psi^T\) will only preserve \(\alpha\) if \(c = 0\), which as we have just seen is not necessary for the non-squeezing property to hold.  
\end{remark}

Recall that in the symplectic case, the affine non-squeezing theorem motivated the definition of the linear symplectic width. The analogous construction for a calibration \(\alpha\) satisfying the conditions of Proposition~\ref{prop:calibns} would be to define, for \(A \subset \mathbb{R}^n\),
\begin{equation}
 \label{eq:4.2}
 w_\alpha(A) = \sup\{\mathcal{V}_k r^k : \psi(B(r)) \subseteq A \text{ for some affine map \(\psi\) with \(\psi^*\alpha = \alpha\)}\},
\end{equation}
where \(\mathcal{V}_k\) is the \(k\)-dimensional volume of the unit ball in \(\mathbb{R}^k\). The \emph{\(\alpha\)-calibrated width} \(w_\alpha\) as defined here clearly has the exact same monotonicity property as the linear symplectic width. The conformality property, which is also clear, takes the form \(w_\alpha(\lambda A) = \lambda^k w_\alpha(A)\). Finally, the normalization property would take the form \(w_\alpha(B(1)) = w_\alpha(Z_L(1)) = \mathcal{V}_k\) for every calibrated subspace \(L\); the proof of the normalization property depends on the affine non-squeezing theorem for \(\alpha\) and is proved in the exact same way as that for the usual symplectic width, as in~\eqref{eq:2.1}.

It is reasonable to ask whether the constant coefficient calibrations on \(\mathbb{R}^n\) that satisfy an affine non-squeezing theorem as in Proposition~\ref{prop:calibns} also satisfy non-linear non-squeezing theorems. We explore the \(\mathrm{G}_2\) case in the next section, where we will see that it is, in fact, trivial.

\subsection{Non-squeezing for particular calibrated geometries}
\label{section:4.2}

Let us recall our short list of examples of calibrations from the previous section:
\begin{enumerate}
 \item The symplectic form \(\omega\) on \(\mathbb{R}^{2n}\), and more generally \(\omega^k/k!\) for \(1 \leq k \leq n\).
 
 \item The associative \(3\)-form $\varphi$ and the coassociative \(4\)-form $\psi$ on \(\mathbb{R}^7\) corresponding to \(\mathrm{G}_2\)-geometry, and the Cayley \(4\)-form \(\Phi\) on \(\mathbb{R}^8\) corresponding to \(\mathrm{Spin}(7)\)-geometry.

 \item The real part of the holomorphic volume form \(\Omega = dz^1 \wedge \cdots \wedge dz^n\) on \(\mathbb{C}^n = \mathbb{R}^{2n}\), and more generally \(\mathrm{Re}(e^{i\theta} \Omega)\) for every real \(\theta\).

 \item The form \(e^1 \wedge \cdots \wedge e^k\) on \(\mathbb{R}^n\), where \(e_1, \dots, e_n\) is any orthonormal basis with dual basis \(e^1, \dots, e^n\).
\end{enumerate}

We saw in Sections~\ref{sec:classical} and~\ref{sec:powers} that the first class of examples all satisfy the non-squeezing theorem, and in particular we explored the notion of a capacity for \(\omega\) and the consequences of the existence of capacities. We also saw that the last example \(e^1 \wedge \cdots \wedge e^k\) satisfies the affine non-squeezing theorem at the end of the previous section (despite it failing to satisfy the hypotheses of Proposition~\ref{prop:calibns}).

Let us now focus on the \(\mathrm{G}_2\) case. Let \(\varphi\) denote the associative \(3\)-form, so that \(\mathrm{G}_2\) is exactly the stabilizer in \(\mathrm{GL}(7, \mathbb{R})\) of \(\varphi\). It is known that \(\mathrm{G}_2\) is closed under transposition. So, in particular, the associative \(3\)-form \(\varphi\) satisfies the affine non-squeezing theorem in the sense of Proposition~\ref{prop:calibns}. In fact, every element of \(\mathrm{G}_2\) preserves the metric, so \(\mathrm{G}_2\) is a subgroup of \(\mathrm{O}(7, \mathbb{R})\). (See, for example,~\cite[Theorem 4.2]{spiro-g2}.)

As we now prove, the non-linear non-squeezing property for \emph{isometric} embeddings of a ball into a cylinder (regardless of whether the cylinder is a calibrated cylinder for some calibration) is straightforward to prove. In particular, it follows that there is a \emph{non-linear} associative non-squeezing theorem.

\begin{proposition} \label{prop:g2trivial}
 Let \(Z_L(R)\) be any cylinder in \(\mathbb{R}^n\) based on any non-trivial subspace \(L\) of \(\mathbb{R}^n\). That is,
 \[
 Z_L(R) := \{ x \in \mathbb{R}^n : |\mathrm{proj}_L(x)|^2 \leq R^2 \}.
 \]
 If \(\varphi \colon B(r) \to \mathbb{R}^n\) is an \emph{isometric} embedding such that \(\varphi(B(r)) \subseteq Z_L(R)\), then \(r \leq R\).
\end{proposition}
\begin{proof}
 Choose an affine line \(\ell\) parallel to \(L\) that intersects \(\varphi(0)\). The line \(\ell\) intersects \(\partial Z_L(R)\) in two points, say, \(p\) and \(q\). The line will also intersect \(\partial(\varphi(B(r))) = \varphi(\partial B(r))\) in two points, say, \(\varphi(x)\) and \(\varphi(y)\) for some \(x\) and \(y\) in \(\partial B(r)\). (See Figure~\ref{fig:4.1}.)
 \begin{figure}[H]
 \centering
\includegraphics[scale=1.2] {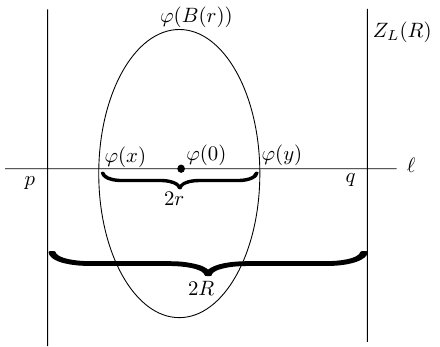}
 \caption{The proof of Proposition~\ref{prop:g2trivial}.}
 \label{fig:4.1}
 \end{figure}
 The segment of \(\ell\) from \(p\) to \(q\) contains the segment from \(\varphi(x)\) to \(\varphi(y)\), so
 \begin{align*}
 2R &\geq d(\varphi(x), \varphi(y)) \\
 &= d(\varphi(x), \varphi(0)) + d(\varphi(0), \varphi(y)) \\
 &= d(x, 0) + d(0, y) \\
 &= 2r. \qedhere
 \end{align*}
\end{proof}

\begin{remark} \label{rmk:isometric}
    Proposition~\ref{prop:g2trivial} shows that whenever the Euclidean metric on \(\mathbb{R}^n\) is determined from a constant coefficient calibration form \(\alpha\), then there is a corresponding non-linear non-squeezing theorem. This includes, for example, the coassociative \(4\)-form \(\psi\) on \(\mathbb{R}^7\) and the Cayley \(4\)-form \(\Phi\) on \(\mathbb{R}^8\).
\end{remark}

We conclude this section with some comments on possible \(\mathrm{G}_2\) (associative) analogues of the applications of non-squeezing and capacities of Section~\ref{section:2.4}. Although the ``non-linear associative non-squeezing'' theorem has a very simple geometric proof, it is not clear to the authors whether this excludes a non-trivial theory of ``associative capacities''. 

In analogy with Section~\ref{section:3.3}, it is natural to ask whether there are associative analogues of Theorems~\ref{thm:nonlinchar} and~\ref{thm:c0closure}. To develop an analogue of Theorem~\ref{thm:nonlinchar}, one would have to in particular determine an associative analogue of the results in Section~\ref{section:2.3} on the symplectic normal forms of ellipsoids. It is not clear to the authors how to do this. 

It is easy to write down as a conjecture the direct analogue of Theorem~\ref{thm:c0closure} for \(\mathrm{G}_2\)-manifolds. Let \((M, g, \varphi)\) be a \(\mathrm{G}_2\)-manifold, where \(\varphi\) is the associative \(3\)-form. Let
\[
 G := \{ F \in \mathrm{Diff}(M) : F^*\varphi = \varphi \}.
\]
It is easy to see that \(G\) is closed in \(\mathrm{Diff}(M)\) with respect to the \(C^1\)-compact-open topology. The question is whether \(G\) is closed in \(\mathrm{Diff}(M)\) with respect to the \(C^0\)-compact-open topology. At least in the case that \(M\) is connected, it would in fact suffice to check whether \(G\) is closed in \(\mathrm{Isom}(M, g)\), since \(\mathrm{Isom}(M, g)\) is closed in \(\mathrm{Diff}(M)\) as a consequence of the Myers-Steenrod theorem that any distance-preserving bijection between connected Riemannian manifolds is in fact metric-preserving. (Indeed, if \((F_n)\) is a sequence of Riemannian isometries converging uniformly on compact sets to a diffeomorphism \(F\), then \(F\) must be a metric isometry since each \(F_n\) is and \(F_n \to F\) pointwise, and then Myers-Steenrod implies that \(F\) is a Riemannian isometry.) Perhaps it is easier to check whether \(G\) is closed in \(\mathrm{Isom}(M, g)\), since the latter is, again by a theorem of Myers-Steenrod, a finite-dimensional Lie group with respect to the \(C^0\)-compact-open topology (whereas \(\mathrm{Diff}(M)\) never is).

\section{(Affine) non-squeezing in the special Lagrangian case} \label{sec:slag}

In Section~\ref{sec:calibrations}, we explored the general question of when a calibration has the ``affine non-squeezing property'', resulting in Proposition~\ref{prop:calibns}. We verified that three of our four main examples of calibrations (the symplectic form \(\omega\) in Theorem~\ref{thm:affns}, its powers \(\omega^k\) in Theorem~\ref{thm:poweraffns}, and the associative \(3\)-form \(\varphi\) in Proposition~\ref{prop:g2trivial}) all had the affine non-squeezing property (as do the coassociative and Cayley calibrations, by Remark~\ref{rmk:isometric}.) In this final section we turn to the fourth and final example, the holomorphic \(n\)-form \(\Omega = dz^1 \wedge \cdots \wedge dz^n\) corresponding to special Lagrangian geometry.

\subsection{A brief review of special Lagrangian geometry}
\label{section:5.1}

We briefly review in this section the results needed in Section~\ref{section:5.2}. We consider \(\mathbb{C}^n = \mathbb{R}^{2n}\) with coordinates \(z^k = x^k + iy^k\), and we define a complex-valued \(n\)-form \(\Omega\) on \(\mathbb{R}^{2n}\) by
\[
 \Omega := dz^1 \wedge \cdots \wedge dz^n = (dx^1 + i dy^1) \wedge \cdots \wedge (dx^n + i dy^n).
\]
Throughout this section, we let \(\alpha\) and \(\beta\) denote the real and imaginary parts of \(\Omega\), respectively. We refer to~\cite{lotay} for the proof of the following theorem.

\begin{theorem} \label{thm:omegacalib}
 (\cite[Theorem 5.2]{lotay}) For all orthonormal \(v_1, \dots, v_n\) in \(\mathbb{R}^{2n}\), 
 \[
 |\Omega(v_1, \dots, v_n)| \leq 1,
 \]
 with equality if and only if \(P = \mathrm{span}\{v_1, \dots, v_n\}\) is Lagrangian. That is, if \(\omega|_P = 0\).
\end{theorem}

Thus the ``\(\Omega\)-calibrated'' subspaces of \(\mathbb{R}^{2n}\) are exactly the Lagrangian subspaces. (Since \(\Omega\) is complex-valued, it is not literally a calibration as defined in Section~\ref{sec:calibrations}, but we will eventually study in the next section an affine non-squeezing theorem for it.) An immediate consequence of Theorem~\ref{thm:omegacalib} is that for every real \(\theta\),
\[
 |\mathrm{Re}(e^{i\theta} \Omega)(v_1, \dots, v_n)| \leq 1 \qquad \text{ for all orthonormal } v_1, \dots, v_n \in \mathbb{R}^{2n}.
\]
To show that each \(\mathrm{Re}(e^{i\theta} \Omega)\) is a calibration, it remains to show that for each \(\theta\), equality holds for at least one \((v_1, \dots, v_n)\). For example,
\[
 (e^{i\theta}\Omega)(e_1, \dots, e_{n-1}, e^{-i\theta}e_n) = \Omega(e_1, \dots, e_n) = 1,
\]
so in particular \(\mathrm{Re}(e^{i\theta}\Omega)(e_1, \dots, e_{n-1}, e^{-i\theta}e_n) = 1\). We call a subspace of \(\mathbb{R}^{2n}\) \emph{special Lagrangian with phase \(e^{-i\theta}\)} if it is calibrated with respect to \(\mathrm{Re}(e^{i\theta}\Omega)\).

\begin{remark} \label{rmk:slag}
It follows from the above discussion and Theorem~\ref{thm:omegacalib} that a subspace of \(\mathbb{R}^{2n}\) is special Lagrangian \emph{for some phase} if and only if it is Lagrangian. Because of this, we will think of \(\Omega\) as a ``complex-valued calibration form'' whose calibrated subspaces are the Lagrangian subspaces.
\end{remark}

In the proof of Proposition~\ref{prop:omegastabilizer} below, we will need the following well-known result, which says that \(\Omega\) determines the complex structure in a precise sense. We include a proof for completeness.

\begin{proposition} \label{prop:complexstr}
Let \(V\) be a \(2n\)-dimensional real vector space. Let \(\Upsilon\) be a complex-valued \(n\)-form on \(V\) that is decomposable and which satisfies \(\Upsilon \wedge \overline{\Upsilon} \neq 0\). Let \(V_\mathbb{C}\) denote the complexification of \(V\). Define
\begin{align*}
V_\mathbb{C}^{1,0} &= \{v \in V_\mathbb{C} : i_v\overline{\Upsilon} = 0\}, \\
V_\mathbb{C}^{0,1} &= \{v \in V_\mathbb{C} : i_v \Upsilon = 0\},
\end{align*}
Then the following statements hold.
\begin{enumerate}
\item \(V_\mathbb{C} = V_\mathbb{C}^{1,0} \oplus V_\mathbb{C}^{0,1}\).
\item The linear map \(J\) on \(V_\mathbb{C}\) defined by letting \(V_\mathbb{C}^{1,0}\) (resp.\ \(V_\mathbb{C}^{0,1}\)) be the \(i\)-eigenspace (resp.\ (\(-i\))-eigenspace) restricts to a complex structure on \(V\).
\end{enumerate}
Moreover, if \(\Upsilon = (e^1 + if^1) \wedge \cdots \wedge (e^n + if^n)\) with respect to a basis \(e_1, \dots, e_n, f_1, \dots, f_n\) of \(V\), then \(Je_k = f_k\) and \(Jf_k = -e_k\) for each \(k\).
\end{proposition}
\begin{proof}
 \,
 \begin{enumerate}
 \item If \(\Upsilon = \theta^1 \wedge \cdots \wedge \theta^n\) for some complex-valued \(1\)-forms \(\theta^1, \dots, \theta^n\), then
 \[
 \theta^1 \wedge \cdots \wedge \theta^n \wedge \overline{\theta^1} \wedge \cdots \wedge \overline{\theta^n} \neq 0.
 \]
 This implies that \((\theta^1, \dots, \theta^n, \overline{\theta^1}, \dots, \overline{\theta^n})\) is a basis of \(V_\mathbb{C}^*\). If the corresponding dual basis of \(V_\mathbb{C}\) is \((v_1, \dots, v_n, \overline{v_1}, \dots, \overline{v_n})\), then we have that \(V_\mathbb{C}^{1,0} = \mathrm{span}\{v_1, \dots, v_n\}\) and \(V_\mathbb{C}^{0,1} = \mathrm{span}\{\overline{v_1}, \dots, \overline{v_n}\}\).

 \item We need to show that if \(v \in V_\mathbb{C}\) is real, then so is \(Jv\). With respect to the decomposition \(V_\mathbb{C} = V_\mathbb{C}^{1,0} \oplus V_\mathbb{C}^{0,1}\), write \(v = v_{1,0} + v_{0,1}\). Since complex conjugation interchanges \(V_\mathbb{C}^{1,0}\) and \(V_\mathbb{C}^{0,1}\), the condition \(v = \overline{v}\) means that \(\overline{v_{1,0}} = v_{0,1}\). It follows that
 \[
 \overline{Jv} = \overline{iv_{1,0} - iv_{0,1}} = -i\overline{v_{1,0}} + i\overline{v_{0,1}} = -iv_{0,1} + iv_{1,0} = Jv.
 \]
 So, \(J\) maps real vectors to real vectors. That is, \(J\) restricts to a complex structure on \(V\).
 \end{enumerate}
 Now if \(\Upsilon = (e^1 + if^1) \wedge \cdots \wedge (e^n + if^n)\), then \(e_k - if_k\) is in \(V_\mathbb{C}^{1,0}\), so
 \[
 Je_k - iJf_k = J(e_k - if_k) = i(e_k - if_k) = f_k + ie_k.
 \]
 Comparing real and imaginary parts gives us \(Je_k = f_k\) and \(Jf_k = -e_k\).
\end{proof}

\subsection{Affine non-squeezing and rigidity}
\label{section:5.2}

According to Proposition~\ref{prop:calibns}, if we can prove that the stabilizer of \(\mathrm{Re}(e^{i\theta} \Omega)\) in \(\mathrm{GL}(2n, \mathbb{R})\) is closed under transposition, then it will satisfy an affine non-squeezing theorem. We examine this for small \(n\) and only for the case that \(e^{i\theta} = 1\).
\begin{enumerate}
 \item \(n = 1\). In this case, \(\mathrm{Re}(\Omega) = e^1\). We explored this example after the proof of Proposition~\ref{prop:calibns}, where we showed that it satisfies an affine non-squeezing theorem, despite its stabilizer failing to be closed under transposition.

 \item \(n = 2\). In this case, \(\mathrm{Re}(\Omega) = e^1 \wedge e^2 - f^1 \wedge f^2\). Up to an orthogonal transformation of \(\mathbb{R}^4\), this is just the symplectic case, which was discussed in Section~\ref{sec:classical}. Indeed, let \(Q \in \mathrm{O}(4)\) be given by \(Qe_1 = e_1\), \(Qe_2 = -f_1\), \(Qf_1 = e_2\), and \(Qf_2 = f_2\). Then \(Q^*(\mathrm{Re}(\Omega)) = \omega\). It follows that the stabilizer of \(\mathrm{Re}(\Omega)\) is conjugate to \(\mathrm{Sp}(4, \mathbb{R})\) by an element of \(\mathrm{O}(4)\), so it is closed under transposition since \(\mathrm{Sp}(4, \mathbb{R})\) is.

 \item \(n = 3\). In this case, it is no longer immediately clear whether the stabilizer of 
 \[
 \mathrm{Re}(\Omega) = e^1 \wedge e^2 \wedge e^3 - e^1 \wedge f^2 \wedge f^3 - f^1 \wedge e^2 \wedge f^3 - f^1 \wedge f^2 \wedge e^3
 \]
 is closed under transposition. It follows from Hitchin's work~\cite[Section 2.2]{hitchin} on \(3\)-forms in six dimensions that if \(A \in \mathrm{GL}(6, \mathbb{R})\), then \(\alpha = \mathrm{Re}(\Omega)\) and \(A^*\alpha\) determine complex structures \(J_\alpha\) and \(J_{A^*\alpha}\) on \(\mathbb{R}^6\), respectively, which are related by
 \[
 J_{A^*\alpha} = \mathrm{sign}(\det A) A^{-1} J_\alpha A = \mathrm{sign}(\det A) A^* (J_\alpha)
 \]
 Moreover, \(J_\alpha\) is the standard complex structure \(J\) on \(\mathbb{R}^6\). It follows from this and Proposition~\ref{prop:omegastabilizer} below that
 \[
 \mathrm{Stab}(\mathrm{Re}(\Omega)) = \mathrm{SL}(3, \mathbb{C}) \cup (\sigma \cdot \mathrm{SL}(3, \mathbb{C})),
 \]
 where \(\sigma\) is complex conjugation. It then follows from Corollary~\ref{cor:transstabilizer} below that this stabilizer is closed under transposition, so \(\mathrm{Re}(\Omega)\) enjoys the affine non-squeezing property.
\end{enumerate}

It is not clear to the authors how to deal with the cases \(n \geq 4\). So it is not clear whether the form \(\mathrm{Re}(\Omega)\), or more generally the forms \(\mathrm{Re}(e^{i\theta} \Omega)\) for real \(\theta\), should satisfy an affine non-squeezing theorem in general.

Rather than trying to approach the general case of the stabilizer of \(\mathrm{Re}(\Omega)\), we can try to focus on, for example, the smaller group given by the stabilizer of both \(\mathrm{Re}(\Omega)\) and \(\mathrm{Im}(\Omega)\). In fact, this is equivalent to studying \(\mathrm{Re}(e^{i\theta_1} \Omega)\) and \(\mathrm{Re}(e^{i\theta_2} \Omega)\) for any two distinct and non-antipodal phases \(e^{i\theta_1}\) and \(e^{i\theta_2}\), as we now show.

\begin{proposition}
 Suppose that \(A \in \mathrm{GL}(2n, \mathbb{R})\) preserves \(\mathrm{Re}(e^{i\theta_1} \Omega)\) and \(\mathrm{Re}(e^{i\theta_2} \Omega)\) for two distinct and non-antipodal phases \(e^{i\theta_1}\) and \(e^{i\theta_2}\). Then \(A\) preserves \(\Omega\).
\end{proposition}
\begin{proof}
 Let us write \(\alpha = \mathrm{Re}(\Omega)\) and \(\beta = \mathrm{Im}(\Omega)\), so that
 \[
 \mathrm{Re}(e^{i\theta_k}\Omega) = (\cos \theta_k) \alpha - (\sin \theta_k) \beta, \qquad k = 1, 2.
 \]
 The equations \(A^*(\mathrm{Re}(e^{i\theta_k}\Omega)) = \mathrm{Re}(e^{i\theta_k}\Omega)\) therefore say that 
 \[
 \begin{pmatrix} \cos \theta_1 & -\sin \theta_1 \\ \cos \theta_2 & -\sin \theta_2\end{pmatrix} \begin{pmatrix} A^*\alpha \\ A^*\beta \end{pmatrix} = \begin{pmatrix} \cos \theta_1 & -\sin \theta_1 \\ \cos \theta_2 & -\sin \theta_2\end{pmatrix} \begin{pmatrix} \alpha \\ \beta \end{pmatrix}.
 \]
 The matrix on each side of the equation has determinant \(\sin(\theta_1 - \theta_2)\), which is non-zero since \(e^{i\theta_1}\) and \(e^{i\theta_2}\) are not antipodal. This implies that \(A^*\alpha = \alpha\) and \(A^*\beta = \beta\), and so \(A^*\Omega = A^*\alpha + iA^*\beta = \alpha + i\beta = \Omega\).
\end{proof}

Recall from Theorem~\ref{thm:omegacalib} that the form \(\Omega\) satisfies the analogous ``comass-one'' condition for a complex-valued form. Once we prove that the stabilizer of \(\Omega\) in \(\mathrm{GL}(2n, \mathbb{R})\) is closed under transposition (this is shown in Corollary~\ref{cor:transstabilizer} below), then the proof of Proposition~\ref{prop:calibns} goes through with \(\alpha\) replaced by \(\Omega\). Note that Proposition~\ref{prop:calibns} does not literally apply because we stated it for real calibrations, but it is easily extended to complex-valued ``calibrations'' with the exact same proof. Recall from Remark~\ref{rmk:slag} that the ``calibrated subspaces'' for \(\Omega\) are precisely the Lagrangian subspaces, which are special Lagrangian subspaces for some phase.

\begin{proposition} \label{prop:omegastabilizer}
 The stabilizer of \(\Omega\) in \(\mathrm{GL}(2n, \mathbb{R})\) is \(\mathrm{SL}(n, \mathbb{C})\), regarded as a subgroup of \(\mathrm{GL}(2n, \mathbb{R})\) via the embedding
 \begin{equation}
 \label{eq:5.2.1}
 \mathrm{GL}(n, \mathbb{C}) \to \mathrm{GL}(2n, \mathbb{R}), \qquad X + iY \mapsto \begin{pmatrix} X & -Y \\ Y & X \end{pmatrix}.
 \end{equation}
In particular, any element $A \in \mathrm{GL}(2n, \mathbb{R})$ that preserves $\Omega$ is necessarily complex-linear.
\end{proposition}
\begin{proof}
 Note that whenever \(A \in \mathrm{GL}(n, \mathbb{C})\), then
 \begin{equation}
 \label{eq:5.2.2}
 A^*\Omega = (\det\nolimits_\mathbb{C} A)\Omega.
 \end{equation}
 It is clear from~\eqref{eq:5.2.2} that every element of \(\mathrm{SL}(n, \mathbb{C})\) preserves \(\Omega\). Conversely, let \(A\) be an element of \(\mathrm{GL}(2n, \mathbb{R})\) that preserves \(\Omega\). The complex structure induced by \(A^*\Omega\) as in Proposition~\ref{prop:complexstr} must be the same as that induced by \(\Omega\). (That is, the standard \(J\).) Since \(A^*e^1, \dots, A^*e^n, A^*f^1, \dots, A^*f^n\) is the dual basis to \(A^{-1}e_1, \dots, A^{-1}e_n, A^{-1}f_1, \dots, A^{-1}f_n\), the last part of the statement of Proposition~\ref{prop:complexstr} implies that \(JA^{-1}e_k = A^{-1}f_k\) and \(JA^{-1}f_k = -A^{-1}e_k\) for every \(k\). In other words, \(AJA^{-1} = J\), so \(A\) is complex-linear. It follows from~\eqref{eq:5.2.2} that \(\det\nolimits_\mathbb{C} A = 1\). That is, \(A \in \mathrm{SL}(n, \mathbb{C})\).
\end{proof}

\begin{corollary} \label{cor:transstabilizer}
 The stabilizer of \(\Omega\) in \(\mathrm{GL}(2n, \mathbb{R})\) is closed under transposition.
\end{corollary}
\begin{proof}
 Let us temporarily denote by \(F\) the embedding in~\eqref{eq:5.2.1}. We need to prove that \(F(\mathrm{SL}(n, \mathbb{C}))\) is closed under transposition. Suppose that \(A = X + iY\) is in \(\mathrm{SL}(n, \mathbb{C})\). Then
 \[
 F(A)^T = \begin{pmatrix} X & -Y \\ Y & X \end{pmatrix}^T = \begin{pmatrix} X^T & Y^T \\ -Y^T & X^T \end{pmatrix} = F(A^*),
 \]
 where \(A^* = X^T - iY^T\) is the conjugate transpose of \(A\). Since \(\det\nolimits_\mathbb{C} A^* = \overline{\det\nolimits_\mathbb{C} A} = 1\), we see that \(F(A)^T\) is in \(F(\mathrm{SL}(n,\mathbb{C}))\).
\end{proof}

Combining Theorem~\ref{thm:omegacalib} and Corollary~\ref{cor:transstabilizer}, we get the following affine non-squeezing theorem for \(\Omega\).

\begin{theorem}[Special Lagrangian affine non-squeezing theorem] \label{thm:slagns}
Let \(\Psi \colon \mathbb{R}^{2n} \to \mathbb{R}^{2n}\) be an affine map of the form \(z \mapsto Az + b\), where \(A\) preserves \(\Omega\) and \(b \in \mathbb{R}^{2n}\). If \(\Psi(B(r)) \subseteq Z_L(R)\), where \(L\) is any Lagrangian subspace of \(\mathbb{R}^{2n}\) (and thus special Lagrangian for some phase), then \(r \leq R\).
\end{theorem}

Following the development in Section~\ref{section:2.2}, it is natural to seek special Lagrangian analogues of the affine rigidity theorem (Theorem~\ref{thm:affrigid}) and of Theorem~\ref{thm:linchar} (a linear map preserves the linear symplectic width of ellipsoids exactly when it is \(\pm\)-symplectic). We must restrict to the maps which preserve all of \(\Omega\), because that is the group of maps for which we were able to establish an affine non-squeezing theorem. 

To state an analogue of Theorem~\ref{thm:affrigid}, we need to first define the analogue of the linear non-squeezing property. We call a subset \(B\) of \(\mathbb{R}^{2n}\) a \emph{linear special Lagrangian ball of radius \(r\)} if there is a linear map \(A\) preserving \(\Omega\) for which \(A(B(r)) = B\). We call a subset \(Z\) of \(\mathbb{R}^{2n}\) a \emph{linear special Lagrangian cylinder} if there is a linear map \(A\) preserving \(\Omega\) for which \(A(Z_L(R)) = Z\) for some \(R > 0\), where \(L = \mathbb{R}^n \subset \mathbb{C}^n\). (Recall the discussion in Remark~\ref{rmk:cylinders} to avoid confusion with this nomenclature.) As in the symplectic case, a consequence of Theorem~\ref{thm:slagns} is that the radius \(R\) is an \(\mathrm{SL}(n, \mathbb{C})\)-invariant of a special Lagrangian cylinder \(Z\).

We say that a linear map \(\Psi \colon \mathbb{R}^{2n} \to \mathbb{R}^{2n}\) has the \emph{special Lagrangian linear non-squeezing property} if, for any \(B\) and \(Z\) as above, the condition \(\Psi(B) \subseteq Z\) implies that \(r \leq R\).

The conclusion of the affine rigidity theorem (Theorem~\ref{thm:affrigid}) in the symplectic case was that if a non-singular map $\Psi$ and its inverse both have the linear non-squeezing property, then $\Psi$ is symplectic or anti-symplectic. We are now working with a complex-valued form, so we replace the condition $\Psi^* \omega = \pm \omega$ with the condition $\Psi^* \Omega = e^{i \theta} \Omega$ for some real \(\theta\) depending on $A$. That is, that the form $\Omega$ is preserved up to a phase. (Note that if this condition on \(\Psi\) holds, then \((e^{-i\theta/n}\Psi)^*\Omega = \Omega\), so that \(e^{-i\theta/n}\Psi \in \mathrm{SL}(n, \mathbb{C})\) by Proposition~\ref{prop:omegastabilizer}. In particular, \(\Psi\) is still necessarily complex-linear.)

Recall that the symplectic affine non-squeezing theorem is also true for affine anti-symplectomorphisms, since the same proof went through under the assumption that \(\omega\) was preserved only up to a sign. In the exact same way, the proof of the ``affine special Lagrangian non-squeezing theorem'' goes through if \(\Omega\) is only preserved up to a phase. Therefore, the corresponding affine rigidity property should be the following.

\emph{Let \(\Psi\) be a non-singular linear map from \(\mathbb{R}^{2n}\) to itself such that both \(\Psi\) and \(\Psi^{-1}\) have the special Lagrangian linear non-squeezing property. Then \(\Psi\) preserves \(\Omega\) up to a phase.}

This is true under the assumption of \emph{complex}-linearity, as Theorem~\ref{thm:slagrid} below shows, but it is false if we merely assume real-linearity. Consider the complex conjugation map \(\sigma \colon \mathbb{C}^n \to \mathbb{C}^n\). If we assume that \(\sigma(B) \subseteq Z\) for some linear special Lagrangian ball \(B\) of radius \(r\) and some linear special Lagrangian cylinder \(Z\) of radius \(R\), then it follows that \((S \circ \sigma \circ T)(B(r)) \subseteq Z_L(R)\) for some \(S\) and \(T\) in \(\mathrm{SL}(n, \mathbb{C})\). Since the ball \(B(r)\) is invariant under \(\sigma\), and \(\sigma\) lies in the normalizer of \(\mathrm{SL}(n, \mathbb{C})\), it follows that \(S \circ (\sigma \circ T \circ \sigma)\) takes \(B(r)\) into \(Z_L(R)\). By Theorem~\ref{thm:slagns}, we have that \(r \leq R\) and so \(\sigma\) has the special Lagrangian linear non-squeezing property, despite not preserving \(\Omega\) up to any phase.

\begin{theorem} \label{thm:slagrid}
 Let \(\Psi\) be a non-singular \emph{complex}-linear map from \(\mathbb{C}^n\) to itself such that both \(\Psi\) and \(\Psi^{-1}\) have the special Lagrangian linear non-squeezing property. Then \(\Psi\) preserves \(\Omega\) up to a phase.
\end{theorem}
\begin{proof}
 As in the proof of Theorem~\ref{thm:affrigid}, assume that \(\Psi^*\Omega \neq e^{i\theta} \Omega\) for any real \(\theta\). Since \(\Psi\) is complex-linear, it follows from~\eqref{eq:5.2.2} that \((\det\nolimits_\mathbb{C} \Psi)\Omega \neq e^{i\theta}\Omega\) for any real \(\theta\). That is, \(\det\nolimits_\mathbb{C} \Psi \not\in S^1\). So, either \(\lvert\det\nolimits_\mathbb{C}\Psi\rvert < 1\) or \(\lvert\det\nolimits_\mathbb{C}\Psi\rvert > 1\). By replacing \(\Psi\) by \(\Psi^{-1}\) if necessary, we may assume that \(\lvert\det\nolimits_\mathbb{C}\Psi\rvert < 1\), so that by~\eqref{eq:5.2.2} again we have that 
 \[
 0 < |\Omega(\Psi^T e_1, \dots, \Psi^T e_n)| < |\Omega(e_1, \dots, e_n)| = 1.
 \]
 Let \(\lambda > 0\) satisfy \(\lambda^n = |\Omega(\Psi^T e_1, \dots, \Psi^T e_n)|\), and let \(\theta\) be such that
 \[
 \Omega(\lambda^{-1} e^{-i\theta} \Psi^T e_1, \dots, \lambda^{-1} e^{-i\theta} \Psi^T e_n) = 1.
 \]
 It follows that there exists a \(\Phi\) in \(\mathrm{SL}(n, \mathbb{C})\) such that \(\Phi e_k = \lambda^{-1} e^{-i\theta} \Psi^T e_k\) for each \(k\). (Define \(\Phi e_k\) as above and then let \(\Phi f_k = J \Phi e_k\).) Let \(A = \Phi^{-1} \Psi^T\), so that \(Ae_k = \lambda e^{i\theta} e_k\) for each \(k\). If \(z \in B(1)\), then
 \[
 \sum_{k=1}^n \langle e_k, A^T z \rangle^2 = \sum_{k=1}^n \langle \lambda e^{i\theta} e_k, z \rangle^2 \leq \lambda^2|z|^2 \leq \lambda^2,
 \]
 where the first inequality above follows from the fact that \(\{e^{i\theta_k}e_k = (\cos \theta)e_k - (\sin \theta) f_k\}_{k=1}^n\) is an orthonormal set. This says that \(A^T\) maps \(B(1)\) into \(Z_L(\lambda)\), where \(L = \mathrm{span}\{e_1, \dots, e_n\}\). But \(\lambda < 1\), so \(A^T = \Psi (\Phi^T)^{-1}\) does not have the special Lagrangian linear non-squeezing property, but then neither does \(\Psi\), which is a contradiction.
\end{proof}

Moving on, in analogy with the linear symplectic width and with the ``calibrated capacity'' defined in~\eqref{eq:4.2}, we define the \emph{linear special Lagrangian width} to be 
\[
 w(A) := \sup\{\mathcal{V}_n r^n : \Psi(B(r)) \subseteq A \text{ for some affine map \(\Psi\) preserving \(\Omega\)}\}.
\]
We want an analogue of Theorem~\ref{thm:linchar} for the linear special Lagrangian width; that is, we want to give a characterization of the ``special Lagrangian linear maps'' (those maps preserving \(\Omega\) up to a phase) as those maps preserving the linear special Lagrangian width. We first need an analogue of Proposition~\ref{prop:antisymp}.

\begin{proposition} \label{prop:slagphasecap}
 If \(\Psi\) is a linear map that preserves \(\Omega\) up to a phase, then \(w(\Psi(A)) = w(A)\) for every subset \(A\) of \(\mathbb{R}^{2n}\).
\end{proposition}
\begin{proof}
 The proof is entirely analogous to the proof of Proposition~\ref{prop:antisymp}. It suffices to show that \(w(A) \leq w(\Psi(A))\) since the other inequality can be established by the same argument with \(\Psi^{-1}\) in place of \(\Psi\). For this, it suffices to show that if \(\Phi(B(r)) \subseteq A\) for some affine \(\Omega\)-preserving map \(\Phi\), then \(\mathcal{V}_n r^n \leq w(\Psi(A))\).

 By assumption, \(\Psi^*\Omega = e^{i\theta} \Omega\) for some real \(\theta\). If \(T\) is the complex-linear map given by multiplication by \(e^{-i\theta/n}\), then the map \(T\) takes \(B(r)\) to \(B(r)\). It follows that \(\Psi \circ \Phi \circ T\) is an \(\Omega\)-preserving affine map which takes \(B(r)\) into \(\Psi(A)\). So, by the definition of \(w\), it follows that \(\mathcal{V}_n r^n \leq w(\Psi(A))\).
\end{proof}

Now, since the affine rigidity theorem for the special Lagrangian case (Theorem~\ref{thm:slagrid}) only holds for complex-linear maps, the following analogue of Theorem~\ref{thm:linchar} must restrict to such maps as well.

\begin{theorem} \label{thm:slagchar}
 Let \(\Psi \colon \mathbb{C}^n \to \mathbb{C}^n\) be a \emph{complex}-linear map. The following are equivalent:
 \begin{enumerate}
 \item \(\Psi\) preserves the linear special Lagrangian width of every ellipsoid centered at the origin.

 \item \(\Psi\) preserves \(\Omega\) up to a phase.
 \end{enumerate}
\end{theorem}
\begin{proof}
 The proof proceeds along the exact same lines as the proof for Theorem~\ref{thm:linchar}. Proposition~\ref{prop:slagphasecap} shows that 2 implies 1, so let us prove that 1 implies 2. By the exact same argument as in the symplectic case (with the linear symplectic width replaced by the linear special Lagrangian width) we see that \(\Psi\) must be invertible. By Theorem~\ref{thm:slagrid}, we want to prove that both \(\Psi\) and \(\Psi^{-1}\) have the special Lagrangian non-squeezing property. Let \(B\) be a linear special Lagrangian ball of radius \(r\) and let \(Z\) be a linear special Lagrangian cylinder of radius \(R\) such that \(\Psi(B) \subseteq Z\). Since \(B\) is an ellipsoid centered at the origin, it follows from the hypothesis 1 that
 \[
 \mathcal{V}_n r^n = w(B) = w(\Psi(B)) \leq w(Z) = \mathcal{V}_n R^n.
 \]
 This implies that \(r \leq R\), proving that \(\Psi\) has the special Lagrangian non-squeezing property. Since \(\Psi^{-1}\) is linear, \(\Psi^{-1}(E)\) is an ellipsoid centered at the origin, and it follows that
 \[
 w(E) = w(\Psi(\Psi^{-1}(E))) = w(\Psi^{-1}(E)),
 \]
 and then it follows from the same argument as for \(\Psi\) that \(\Psi^{-1}\) also has the special Lagrangian non-squeezing property. So, by Theorem~\ref{thm:slagrid}, the map \(\Psi\) preserves \(\Omega\) up to a phase. (This is the only step in this proof where we use the complex linearity of \(\Psi\); we need it to be able to invoke Theorem~\ref{thm:slagrid}.)
\end{proof}

\section{Further questions}
\label{sec:conclusion}

There are several further questions which arise naturally. First, in the particular context of the special Lagrangian calibration, we have the following questions:

\begin{enumerate}
    \item We established an affine rigidity result (Theorem~\ref{thm:slagrid}) and a characterization of width-preserving maps (Theorem~\ref{thm:slagchar}) in the special Lagrangian case, for necessarily complex-linear maps. One could also try to develop a non-linear theory along the lines of Section~\ref{section:2.4} for such maps.
    
    \item We have left open the questions of whether the stabilizer of \(\mathrm{Re}(\Omega)\) (and more generally the stabilizer of \(\mathrm{Re}(e^{i\theta}\Omega)\) for real \(\theta\)) is closed under transposition for the cases \(n \geq 4\).
    
    \item The non-linear results in Section~\ref{section:2.4} relied on Gromov's non-squeezing theorem (Theorem~\ref{thm:gns}) being true; more specifically, the existence of a normalized symplectic capacity was equivalent to Gromov's theorem. It seems one would need a special Lagrangian analogue of the non-squeezing theorem to develop analogues of the results of Section~\ref{section:2.4}.
    
    \item To our knowledge, the proofs of Gromov's non-squeezing theorem that establish the existence of a normalized capacity all rely on the fact that symplectic manifolds admit Hamiltonian vector fields and therefore a theory of Hamiltonian dynamics. It is not clear what the analogue would be for the special Lagrangian case. On the other hand, Gromov's original proof of the non-squeezing theorem using pseudo-holomorphic curves may be suitable for adaptation to a proof of a non-squeezing theorem for the special Lagrangian case. The notion of a \emph{Smith immersion} generalizes pseudo-holomorphic curves to other calibrated geometries, and it has been shown that Smith immersions enjoy many analytic properties analogous to the classical analytical properties of pseudo-holomorphic curves. (See~\cite{CKM} and~\cite{antonspiro}.)
\end{enumerate}

Recall that in Section~\ref{sec:calibrations}, we demonstrated that an affine non-squeezing theorem is trivial for a calibration \(\alpha\) on \(\mathbb{R}^n\) whose stabilizer group is contained in \(\mathrm{SO}(n)\), because the affine maps preserving \(\alpha\) are isometries. Nevertheless, it may still be possible to develop a non-trivial theory of ``\(\alpha\)-capacities'' for such calibrations, and to ask other related questions, such as \(\alpha\)-analogues of Theorems~\ref{thm:nonlinchar} and~\ref{thm:c0closure}, as we discussed at the end of Section~\ref{sec:calibrations}.

More generally, the authors would like to call attention to the fact that, in analogy with our philosophy in the present paper, there are many interesting questions that can be considered for general calibrations, which are to date much better understood in the classical context of the K\"ahler calibration. Some examples are the following:

\begin{itemize}
    \item In~\cite{CKM-variational}, a variational characterization of calibrated submanifolds is established for associative and coassociative calibrations, generalizing a result for the classical K\"ahler calibration. Surprisingly, the Cayley case behaves quite differently in this context. The reasons for this are still mysterious.
    
    \item In~\cite{KM-extrinsic}, the notion of a \emph{compliant} calibration was defined. A calibration on \(\mathbb{R}^n\) is compliant if its stabilizer in \(\mathrm{SO}(n)\) satisfies a particular Lie-theoretic property. In particular, all of the geometrically interesting calibrations are compliant. It would be interesting to see if and how this condition interacts with the ideas considered in the present paper.
    
    \item A particularly rich and well-studied \emph{pluripotential theory} exists for the K\"ahler calibration. Work by Harvey--Lawson~\cite{HL-plu} has shown that much of this theory extends to other calibrations, but there is a great deal that remains to be better understood. See also~\cite{PR} for an application of these ideas to analogues of the \(\partial \bar{\partial}\)-lemma for other calibrated geometries.
    
    \item Other exciting recent work is~\cite{IP}, which considers analogues of Liouville's Theorem for calibrated geometries, and~\cite{BIM}, which considers analogues of the Schwarz Lemma and hyperbolicity for calibrated geometries. Both papers are also related to quasi-conformal geometry.
\end{itemize}

We hope that our work and the work described above will spur further interest in these ideas.



\bibliographystyle{plain}

\phantomsection 

\renewcommand*{\bibname}{References}

\addcontentsline{toc}{section}{References}

\bibliographystyle{plain}
\bibliography{non-squeezing}


\end{document}